\documentclass[11pt]{article}
\usepackage{amsmath, amssymb,verbatim,appendix}
\usepackage[mathscr]{eucal}
\usepackage{amscd}
\usepackage{amsthm}
\usepackage{enumerate}
\usepackage{cite}
\usepackage[margin=1in]{geometry}

\newtheorem{theorem}{Theorem}
\newtheorem{lemma}{Lemma}
\newtheorem{corollary}{Corollary}
\newtheorem{proposition}{Proposition}

\newtheorem{conjecture}{Conjecture}

\newtheorem{claim}{Claim}

\theoremstyle{definition}
\newtheorem{definition}{Definition}

\newcommand{\Ybar}{\overline{Y}}

\newcommand{\calF}{\mathcal{F}}
\newcommand{\calH}{\mathcal{H}}

\newcommand{\calP}{\mathcal{P}}

\newcommand{\calG}{\mathcal{G}}

\newcommand{\calS}{\mathcal{S}}

\newcommand{\tildeG}{\tilde{G}}

\newcommand{\tildeV}{\tilde{V}}

\newcommand{\tildeX}{\tilde{X}}
\newcommand{\tildeY}{\tilde{Y}}
\newcommand{\tildeZ}{\tilde{Z}}

\newcommand{\exs}{{\rm{ex}}_{\Sigma}}
\newcommand{\expi}{{\rm{ex}}_{\Pi}}

\title{
An extremal graph problem with a transcendental solution}

\author{
Dhruv Mubayi \footnote{Department of Mathematics, Statistics, and Computer Science, University of Illinois at 
Chicago. Research supported in part by NSF Grant DMS 1300138; mubayi@uic.edu}
\and
Caroline Terry \footnote{Department of Mathematics, Statistics, and Computer Science, 
University of Illinois at Chicago; 
cterry3@uic.edu}
}

\begin{document}

\maketitle

\begin{abstract}
We prove that the number of multigraphs with vertex set $\{1, \ldots, n\}$ such that every four vertices span at most nine edges is 
$a^{n^2 + o(n^2)}$
where $a$ is  transcendental (assuming Schanuel's conjecture from number theory). 
This is an easy consequence of the solution to a related problem about maximizing the product of the edge multiplicities  in certain multigraphs, and appears to be the first explicit (somewhat natural) question in extremal graph theory
whose solution is  transcendental. These results may  shed light on a question of Razborov who asked whether there are conjectures or theorems in extremal combinatorics which cannot be proved by a certain class of finite methods that include Cauchy-Schwarz arguments. 

 Our proof involves  a novel application of Zykov symmetrization applied to multigraphs, a rather technical progressive induction, and a  straightforward use of hypergraph containers.
\end{abstract}

\section{Introduction}

All logarithms in this paper are natural logarithms unless the base is explicitly written.
Given a set $X$ and a positive integer $t$, let ${X\choose t}=\{Y\subseteq X: |Y|=t\}$.  A \emph{multigraph} is a pair $(V,w)$, where $V$ is a set of vertices and $w:{V\choose 2}\rightarrow \mathbb{N}=\{0,1,2, \ldots\}$. 

\begin{definition} Given integers $s\geq 2$ and $q\geq 0$, a multigraph $(V,w)$ is an \emph{$(s,q)$-graph} if for every $X\in {V\choose s}$ we have $\sum_{xy\in {X\choose 2}}w(xy)\leq q$.  An $(n,s,q)$-graph is an $(s,q)$-graph with $n$ vertices, and $F(n,s,q)$ is the set of $(n,s,q)$-graphs with vertex set $[n]:=\{1,\ldots, n\}$. \end{definition}
  The main goal of this paper is to prove that the maximum product of  the edge multiplicities over all $(n,4,15)$-graphs is
\begin{equation} \label{P} 2^{\gamma n^2 +O(n)}\end{equation}
where 
$$ \gamma= \frac{\beta^2}{2}+\beta(1-\beta)\frac{\log3}{\log 2} \qquad \hbox{  and } \qquad \beta=\frac{\log3}{2\log 3-\log 2} \,.$$ It is an easy exercise to show that both $\beta$ and $\gamma$ are  transcendental by the Gelfond-Schneider theorem~\cite{gelfond}.
Using (\ref{P}), we will prove that
\begin{equation} \label{F} |F(n,4,9)|=a^{n^2+o(n^2)},\end{equation}
 where  $a=2^{\gamma}$ is also transcendental (assuming Schanuel's conjecture in number theory).

Due to the Erd\H os-Simonovits-Stone theorem~\cite{ErdosStone, ErdSim1}, many natural extremal graph problems involving edge densities have rational solutions, and their enumerative counterparts have algebraic solutions.
 For example, the Erd\H os-Kleitman-Rothschild theorem~\cite{EKR}  states that the number of triangle-free graphs on $[n]$ is
 $2^{n^2/4+o(n^2)}$ and $2^{1/4}$ is algebraic since $1/4$ is rational.    For hypergraphs the situation is more complicated, and the first author and Talbot~\cite{MT} proved that certain partite hypergraph Tur\'an problems have irrational solutions. Going further,
the question of obtaining transcendental solutions for natural extremal problems is an intriguing one.  This was perhaps first explicitly posed by Fox (see  \cite{pikhurko}) in the context of Tur\'{a}n densities of hypergraphs.  Pikhurko \cite{pikhurko} showed that the set of hypergraph Tur\'{a}n densities is uncountable, thereby proving the existence of transcendental ones (see also \cite{grosu}), but his list of forbidden hypergraphs was infinite.  When only finitely many hypergraphs are forbidden, he obtained irrational densities.    To our knowledge, (\ref{P}) and (\ref{F}) are the first examples of  fairly natural extremal graph problems whose answer is given  by (explicitly defined) transcendental numbers (modulo Schanuel's conjecture in the case of (\ref{F})).

Another area that (\ref{P}) may shed light on is the general question of whether certain proof methods suffice to solve problems in extremal combinatorics. The explosion of results in extremal combinatorics using Flag Algebras~\cite{Razborov} in recent years has put the spotlight on such questions, and Razborov first posed this  in~(Question 1, \cite{Razborov}). A significant result in this direction is due to Hatami and Norine~\cite{HN}.  They  prove that the related question (due in different forms to Razborov, Lov\'asz, and Lov\'asz-Szegedy) of whether every true linear inequality between homomorphism densities can be proved using a finite amount of manipulation with homomorphism densities of finitely many graphs is not decidable. While  we will not attempt to state~(Question 1, \cite{Razborov}) rigorously here, its  motivation  is to understand whether the  finite methods that are typically used in combinatorial proofs of extremal results (formalized by Flag Algebras and the Cauchy-Schwarz calculus) suffice for all extremal problems involving subgraph densities. Although we cannot settle this, one might speculate that these finite ``Cauchy-Schwarz methods" may not be enough to obtain (\ref{P}). In any event, (\ref{P}) seems to be a good test case. Curiously, our initial explorations into (\ref{P}) were through Flag Algebra computations which gave the answer to several decimal places and motivated us to obtain sharp results, though our eventual proof of (\ref{P}) uses no Flag Algebra machinery. Instead, it uses some novel extensions of classical methods in extremal graph theory, and we expect that these ideas will be used to solve other related problems. 

 As remarked earlier, 
(\ref{F}) is a fairly straightforward consequence of (\ref{P}) and since the expression in (\ref{P}) is obtained as a  product (rather than sum) of numbers, it is easier to obtain a transcendental number in this way. However, we should point out that an extremal example for (\ref{P}) (and possibly all extremal examples, though we were not able to show this) involves partitioning the vertex set $[n]$ into two parts where one part has size approaching $\beta n$, and $\beta$ is also transcendental (see Definition~\ref{W} and Theorem~\ref{caseiv} in the next section). This might indicate the difficulty  in proving (\ref{P}) using the sort of finite methods discussed above.

Finally, we  would like to mention that the problem of asymptotic enumeration of $(n,s,q)$-graphs is a natural extension of the work on extremal problems related to $(n,s,q)$-graphs by Bondy and Tuza in \cite{bondytuza} and by F\"{u}redi and K\"{u}ndgen in \cite{furedikundgen}.  Further work in this direction, including a systematic investigation of  extremal, stability, and enumeration results for a large class of pairs $(s,q)$, will appear in~\cite{MT3} (see also~\cite{MTmetric} for another example on multigraphs).
Alon~\cite{Alon} asked whether  the transcendental behavior witnessed here is an isolated case. Although we believe that there are infinitely many such examples (see Conjecture~\ref{a} in Section~\ref{con}) we were not able to prove this for any other pair $(s,q)$. The infinitely many pairs for which we obtain precise extremal results in~\cite{MT3} have either rational or integer densities. 
\section{Results}\label{sectionmainresults}

Given a multigraph $G=(V,w)$, define $P(G)=\prod_{xy\in {V\choose 2}}w(xy)$ and $S(G)=\sum_{xy\in {V\choose 2}}w(xy)$. 

\begin{definition} Suppose $s\geq2$ and $q\geq 0$ are integers.  Define
$$
\expi(n,s,q)=\max \{P(G): G\in F(n,s,q)\},\quad \hbox{ and }\quad \expi(s,q)=\lim_{n\rightarrow \infty} \Big(\expi(n,s,q)\Big)^{\frac{1}{{n\choose 2}}}.
$$
An $(n,s,q)$-graph $G$ is \emph{product-extremal} if $P(G)=\expi(n,s,q)$.  The limit $\expi(s,q)$ (which we will show always exists) is called the \emph{asymptotic product density}. 
\end{definition}

\noindent Our first main result is an enumeration theorem for $(n,s,q)$-graphs in terms of $\expi(s,q+{s\choose 2})$.

\begin{theorem}\label{counting}  Suppose $s\geq 2$ and $q\geq 0$ are integers. If $\expi(s, q+{s\choose 2})>1$, then  
$$
\expi\Big(s,q+{s\choose 2}\Big)^{{n\choose 2}}\leq |F(n,s,q)|\leq \expi\Big(s,q+{s\choose 2}\Big)^{(1+o(1)){n\choose 2}},
$$
and if $\expi(s,q+{s\choose 2})\leq 1$, then $|F(n,s,q)|\leq 2^{o(n^2)}$.
\end{theorem}

Theorem \ref{counting} will be proved in Section \ref{containers} using the hypergraph containers method of \cite{Baloghetal1, saxton-thomason} along with a multigraph version of the graph removal lemma. Theorem \ref{counting} reduces the problem of enumerating $F(n,4,9)$ to computing $\expi(4,15)$.  This will be the focus of our remaining results. 

\begin{definition} \label{W}
Given $n$, let $W(n)$ be the set of multigraphs $G=([n],w)$ for which there is a partition $L, R$ of $[n]$ such that $w(xy)=1$ if $xy\in {L\choose 2}$, $w(xy)=2$ if $xy\in \binom{R}{2}$, and $w(xy)=3$ if $(x,y)\in L\times R$.
\end{definition}

\noindent Notice that $W(n)\subseteq F(n,4,15)$ for all $n\in \mathbb{N}$.  Straightforward calculus shows that for $G\in W(n)$, the product $P(G)$ is maximized when $|R|\approx \beta n$, where $\beta=\frac{\log3}{2\log 3-\log 2}$ is a transcendental number. This might indicate the difficulty of obtaining this extremal construction using a standard induction argument.  Given a family of hypergraphs 
$\cal F$, write $\calP({\cal F})$ for the set of $G\in {\cal F}$ with $P(G)=\max\{P(G'):G'\in {\cal F}\}$. Use the shorthand $\calP(n,s,q)$ for $\calP(F(n,s,q))$.
\begin{theorem}\label{caseiv}
For all sufficiently large $n$, $\calP(W(n))\subseteq \calP(n,4,15)$.  Consequently
\begin{align*}
\expi(n, 4,15)=\max_{G \in W(n)}P(G)=2^{\gamma n^2+O(n)} \qquad \hbox { and } \qquad \expi(4,15)=2^{2\gamma},
\end{align*}
where $\gamma= \beta^2/2+\beta(1-\beta)\log_23$ and $\beta=\frac{\log3}{2\log 3-\log 2}$. 
\end{theorem}

\noindent For reference, $\beta\approx .73$ and $2^\gamma \approx 1.49$.  The  result below follows directly from Theorems \ref{counting} and \ref{caseiv}.

\begin{theorem}\label{mainthm}
$|F(n,4,9)|=2^{\gamma n^2+o(n^2)}$.
\end{theorem}
\begin{proof}
Theorem \ref{caseiv} implies  $\expi(4,15)=2^{2\gamma}>1$.  By Theorem \ref{counting}, this implies that 
$$\expi(4,15)^{n\choose 2}\leq |F(n,4,9)|\leq \expi(4,15)^{(1+o(1)){n\choose 2}}.$$  Consequently,  $|F(n,4,9)|=2^{2\gamma{n\choose 2}+o(n^2)}=2^{\gamma n^2+o(n^2)}$.
\end{proof}

Recall that Schanuel's conjecture from the 1960s (see~\cite{Lang}) states the following: if $z_1,\ldots, z_n$ are complex numbers which are linearly independent over $\mathbb{Q}$, then $\mathbb{Q}(z_1,\ldots, z_n, e^{z_1},\ldots, e^{z_n})$ has transcendence degree at least $n$ over $\mathbb{Q}$.  As promised in the introduction and abstract, we now show that assuming Schanuel's conjecture, $2^{\gamma}$ is transcendental.  Observe that this implies $\expi(4,15)=2^{2\gamma}$ is also transcendental over $\mathbb{Q}$, assuming Schanuel's conjecture.
\begin{proposition}
Assuming Schanuel's conjecture, $2^{\gamma}$ is transcendental.
\end{proposition}
\begin{proof}
Assume Schanuel's Conjecture  holds.  It is well-known that Schanuel's conjecture implies $\log 2$ and $\log 3$ are algebraically independent over $\mathbb{Q}$ (see for instance \cite{Waldschmidt2015}). Observe $\gamma=\frac{f(\log 2,\log 3)}{g(\log 2,\log 3)}$ where $f(x,y)=xy^2/2+y^2(y-x)$ and $g(x,y)=x(2y-x)^2$.  Note the coefficient of $x^3$ in $f(x,y)$ is $0$ while in $g(x,y)$ it is $1$.  We now show $\log 2, \log 3, \gamma \log 2$ are linearly independent over $\mathbb{Q}$.  Suppose towards a contradiction that this is not the case.  Then there are non-zero rationals $p,q,r$ such that 
$$
p\log 2+q\log 3+r\gamma \log 2=0.
$$
Replacing $\gamma$ with $\frac{f(\log 2,\log 3)}{g(\log 2,\log 3)}$, this implies $p\log 2+q\log 3+r\frac{f(\log 2,\log 3)}{g(\log 2,\log 3)} \log 2=0$. By clearing the denominators of $p,q,r$ and multiplying by $g(\log 2,\log 3)$, we obtain that there are non-zero integers $a,b,c$ such that
$$
(a\log 2+b\log 3)g(\log 2,\log 3)+cf(\log 2,\log 3)\log 2=0.
$$
Let $p(x,y)=(ax+by)g(x,y)+cf(x,y)x$.  Observe that $p(x,y)$ is a rational polynomial such that $p(\log 2,\log 3)=0$.  Since the coefficient of $x^3$ is $1$ in $g(x,y)$ and $0$ in $f(x,y)$, the coefficient of $x^4$ in $p(x,y)$ is $a\neq 0$.  Thus $p(x,y)$ has at least one non-zero coefficient, contradicting that $\log 2$ and $\log 3$ are algebraically independent over $\mathbb{Q}$.  Thus $\log 2, \log 3, \gamma \log 2$ are linearly independent over $\mathbb{Q}$, so Schanuel's conjecture implies $\mathbb{Q}(\log 2,\log 3, \gamma \log 2, 2^{\gamma})$ has transcendence degree at least $3$ over $\mathbb{Q}$.  Suppose towards a contradiction that $2^{\gamma}$ is not transcendental.  Then $\log 2, \log 3, \gamma \log 2$ must be algebraically independent over $\mathbb{Q}$.  Let $h(x,y,z)=zg(x,y)-xf(x,y)$.  Then it is clear $h(x,y,z)$ has non-zero coefficients, and 
$$
h(\log 2, \log 3, \gamma \log 2)=(\gamma \log 2)g(\log2,\log3)-(\log 2) f(\log2, \log3)=0,
$$
where the second equality uses the fact that $\gamma=\frac{f(\log 2,\log 3)}{g(\log 2,\log 3)}$.  But this implies $\log 2, \log 3, \gamma \log 2$ are algebraically dependent over $\mathbb{Q}$, a contraction.  Thus $2^{\gamma}$ is transcendental. 
\end{proof}

\section{General enumeration in terms of asymptotic product density}
In this section we prove Theorem \ref{counting}, our general enumeration theorem for $(n,s,q)$-graphs.  We will use a version of the hypergraph containers theorem (Balogh-Morris-Samotij \cite{Baloghetal1}, Saxton-Thomason \cite{saxton-thomason}), a graph removal lemma for edge-colored graphs, and Proposition \ref{density} below, which shows $\expi(s,q)$ exists for all $s\geq 2$ and $q\geq 0$.  Given $G=(V,w)$ and $X\subseteq V$, let $G[X]=(X,w\upharpoonright_{{X\choose 2}})$.  

\begin{proposition}\label{density}
For all $n\geq s\geq 2$ and $q\geq 0$, $\expi(s,q)$ exists and $\expi(n,s,q)\geq \expi(s,q)^{n\choose 2}$.  If $q\geq {s\choose 2}$, then $\expi(s,q)\geq 1$.
\end{proposition}
\begin{proof}
Fix $s\geq 2$ and $q\geq 0$.  Clearly, for all $n\geq s$, $b_n:=(\expi(n,s,q))^{\frac{1}{{n\choose 2}}}\geq 0$.  We now show the $b_n$ are non-increasing.  For $n>s$ and $G\in F(n,s,q)$, note
\begin{align*}
P(G) = \Big(\prod_{i\in [n]}P(G[[n]\setminus \{i\}])\Big)^{1/(n-2)}\leq \Big(\prod_{i\in [n]}b_{n-1}^{n-1\choose 2}\Big)^{1/(n-2)}= b_{n-1}^{n{n-1\choose 2}/n-2}= b_{n-1}^{n\choose 2}.
\end{align*}
Therefore, for all $G\in F(n,s,q)$, $P(G)^{1/{n\choose 2}}\leq b_{n-1}$, so $b_n\leq b_{n-1}$ and $\lim_{n\rightarrow \infty}b_n=\expi(s,q)$ exists.  The inequality $\expi(n,s,q)\geq \expi(s,q)^{n\choose 2}$ follows because the $b_n$ are non-increasing.  If $q\geq {s\choose 2}$, then for all $n\geq s$, the multigraph $G=([n],w)$, where $w(xy)=1$ for all $xy\in {[n]\choose 2}$, is in $F(n,s,q)$.  This shows $b_n\geq 1$ for all $n\geq s$, so by definition, $\expi(s,q)\geq 1$.  
\end{proof}

We now state a version of the hypergraph containers theorem.  Specifically, Theorem \ref{containers} below is a simplified version of Corollary 3.6 from \cite{saxton-thomason}. We first require some notation.  Given $r\geq 2$, an \emph{$r$-uniform hypergraph} is a pair $H=(W,E)$ where $W$ is a set of vertices and $E\subseteq {W\choose r}$ is a set of edges.  Given $C\subseteq W$, $H[C]$ is the hypergraph $(C,E\cap{C\choose r})$.  The \emph{average degree} of $H$ is $d=r|E|/|W|$, and $e(H)=|E|$ is the number of edges in $H$.  Given a set $X$, $2^X$ denotes the power set of $X$.

\begin{definition}Suppose $H=(W,E)$ is an $r$-uniform hypergraph with average degree $d$ and fix $\tau >0$.  For every $\sigma \subseteq W$, $x\in W$, and $j\in [r]$, set 
$$
d(\sigma) = |\{e\in E: \sigma \subseteq e\}|\qquad \hbox{ and }\qquad d^{(j)}(x)=\max \{ d(\sigma): x\in \sigma \subseteq W, |\sigma|=j\}.
$$
If $d>0$, then for each $j\in [r]$, define $\Delta_j=\Delta_j(\tau)$ to satisfy the equation
$$
\Delta_j \tau^{j-1}nd=\sum_{x\in W}d^{(j)}(x)\qquad \hbox{ and set }\qquad \Delta(H,\tau)=2^{{r\choose 2}-1}\sum_{j=2}^{r} 2^{-{j-1\choose 2}}\Delta_j.
$$
If $d=0$, set $\Delta(H,\tau)=0$.  The function $\Delta(H,\tau)$ is called the \emph{co-degree function}.
\end{definition}

\begin{theorem}[Corollary 3.6 from \cite{saxton-thomason}]\label{containers}
Fix $0<\epsilon, \tau<\frac{1}{2}$.  Suppose $H$ is an $r$-uniform hypergraph with vertex set $W$ of size $N$ satisfying $\Delta(H,\tau)\leq \frac{\epsilon}{12r!}$.  Then there exists a positive constant $c=c(r)$ and a collection $\mathcal{C}\subseteq 2^{W}$ such that the following holds. 
\begin{enumerate}[(i)]
\item For every independent set $I$ in $H$, there is some $C\in \mathcal{C}$, such that $I\subseteq C$.
\item For all $C\in \mathcal{C}$, we have $e(H[C])\leq \epsilon e(H)$.
\item $\log |\mathcal{C}| \leq c\log(1/\epsilon) N\tau \log (1/\tau)$.
\end{enumerate}
\end{theorem}

Our next goal is to prove a version of Theorem \ref{containers} for multigraphs.  Suppose $G=(V,w)$ is a multigraph.  For all $xy\in {V\choose 2}$, we will refer to $w(xy)$ as the \emph{multiplicity} of $xy$.  The \emph{multiplicity of $G$} is $\mu(G)=\max\{w(xy): xy\in {V\choose 2}\}$. Given another multigraph, $G'=(V',w')$, we say that $G$ is a  \emph{submultigraph} of $G'$ if $V=V'$ and for each $xy\in {V\choose 2}$, $w(xy)\leq w'(xy)$. 

\begin{definition}
Suppose $s\geq 2$ and $q\geq 0$ are integers.  Set 
$$
\calH(s,q)=\{G=([s],w):\mu(G)\leq q\text{ and }S(G)>q\},\quad \text{ and }\qquad g(s,q)=|\calH(s,q)|.
$$
If $G=(V,w)$ is a multigraph, let $\calH(G,s,q)=\{X\in {V\choose s}: G[X]\cong G'\text{ some }G'\in \calH(s,q)\}$. 
\end{definition}

Observe $G=(V,w)$ is an $(s,q)$-graph if and only if $\calH(G,s,q)=\emptyset$.   Suppose $n$ is an integer.  We now give a procedure for defining a hypergraph $H(n)=(W,E)$.  Set $[0,q]=\{0,1,\ldots, q\}$. The vertex set of $H(n)$ is $W=\{(f,u): f\in {[n]\choose 2}, u\in [0,q]\}$.  The idea is that each pair $(f,u)$ corresponds to the choice ``the multiplicity of $f$ is $u$.''  The edge set $E$ of $H(n)$ consists of all sets of the form $\{(f,w(f)): f\in {A\choose 2}\}$, where $A\subseteq[n]$ and $(A,w)\cong G$ for some $G\in \calH(s,q)$.  For any $\sigma \subseteq W$, define $V(\sigma)$ be the set of all $i\in [n]$ appearing in an element of $\sigma$, i.e.,
$$
V(\sigma)=\Big\{i\in [n]: \text{ there is some }(f,u)\in \sigma \text{ with }i\in f\Big\}.
$$
Observe that for all $e\in E$, $|e|={s\choose 2}$, $|V(e)|=s$, and for all $f\in {V(e)\choose 2}$, there is exactly one $u\in [0,q]$ such that $(f,u)\in e$.  By definition, $H(n)$ is an ${s\choose 2}$-uniform hypergraph, $|W|=(q+1){n\choose 2}$, and $|E|=g(s,q){n\choose s}$.  We now prove our multigraph version of Theorem \ref{containers}.  Many of the computations in the proof are modeled on those used to prove Corollary 3 in \cite{Stegeretal}

\begin{theorem}\label{containerscor}
For every $0<\delta<1$ and integers $s\geq 2$, $q\geq 0$, there is a constant $c=c(s,q,\delta)>0$ such that the following holds.  For all sufficiently large $n$, there is $\mathcal{G}$ a collection of multigraphs of multiplicity at most $q$ and with vertex set $[n]$ such that
\begin{enumerate}[(i)]
\item for every $J\in F(n,s,q)$, there is $G\in \mathcal{G}$ such that $J$ is a submultigraph of $G$,
\item for every $G\in \mathcal{G}$, $|\calH(G,s,q)|\leq \delta {n\choose s}$, and 
\item $\log |\mathcal{G}| \leq cn^{2-\frac{1}{4s}}\log n$.
\end{enumerate}
\end{theorem}

\begin{proof}
Clearly it suffices to show Theorem \ref{containerscor} holds for all $0<\delta<1/2$. Fix $0<\delta<1/2$ and integers $s\geq 2$, $q\geq 0$.  Let $c_1=c_1({s\choose 2})$ be from Theorem \ref{containers}, and set $c=c(s,q,\delta)=\frac{c_1}{4s}\log(\frac{g(s,q)}{\delta})(q+1)$.  Assume $n$ is sufficiently large. We show there is a collection $\calG$ of multigraphs with multiplicity at most $q$ and vertex set $[n]$ such that (i)-(iii) hold for this $c$ and $\delta$.  Let $H:=H(n)$ be the ${s\choose 2}$-uniform hypergraph described above.  In particular, $H=(W,E)$ where  
\begin{align*}
W&=\{(f,u): f\in {[n]\choose 2}, u\in [0,q]\}\text{ and }\\
E&=\Big\{ \Big\{(f,w(f)): f\in {A\choose 2}\Big\}: A\subseteq [n]\text{ and }(A,w)\cong G,\text{ for some }G\in \calH(s,q)\Big\}.
\end{align*}
Set $\epsilon=\frac{\delta}{g(s,q)}$, $\tau=n^{\frac{-1}{4s}}$, and $N=|W|$.  We show the hypotheses of Theorem \ref{containers} are satisfied by $H$ with this $\epsilon$ and $\tau$.  Since $n$ is sufficiently large, $0<\tau<1/2$.  By definition of $\epsilon$, $0<\epsilon\leq \delta<1/2$.  We must now verify that $\Delta(H,\tau)\leq \frac{\epsilon}{12r!}$.  We begin with bounding the $\Delta_j$.  Fix $2\leq j\leq {s\choose 2}$ and $\sigma\subseteq W$ with $|\sigma|=j$.  We claim 
\begin{align}\label{dsigma}
d(\sigma) \leq g(s,q) n^{s-\frac{1}{2}-\sqrt{2j}}.
\end{align}
If $|V(\sigma)|>s$, then because every $e\in E$ satisfies $|V(e)|=s$, we must have $d(\sigma)=0$, so (\ref{dsigma}) holds.  Similarly, if there are $u\neq v\in [0,q]$ and $f\in {V(\sigma)\choose 2}$ such that $(f,u), (f,v)\in \sigma$, then because no $e\in E$ can contain both $(f,u)$ and $(f,v)$, we must have $d(\sigma)= 0$, so (\ref{dsigma}) holds.  Assume now $|V(\sigma)|\leq s$ and for all $f\in {V(\sigma)\choose 2}$, there is at most one $u\in [0,q]$ such that $(f,u)\in \sigma$.  Note this implies $|\sigma|\leq {|V(\sigma)|\choose 2}$. For each $(f,u)$ in $\sigma$, set $w(f)=u$.  Then $w$ is a (possibly partial) function from ${V(\sigma)\choose 2}$ into $[0,q]$.  Every edge $e$ containing $\sigma$ can be constructed as follows.
\begin{itemize}
\item Choose an $s$-element subset $X\subseteq [n]$ extending $V(\sigma)$.  There are ${n-|V(\sigma)|\choose s-|V(\sigma)|}$ ways to do this.
\item Extend $w$ to a total function ${X\choose 2}\rightarrow [0,q]$ so that $(X,w)\cong G$ for some $G\in \calH(s,q)$, and set $e=\{(f,w(f)): f\in {X\choose 2}\}$.  There are at most $g(s,q)$ ways to do this.
\end{itemize}
This shows that
\begin{align}\label{ineq1}
d(\sigma) \leq g(s,q) {n-|V(\sigma)|\choose s-|V(\sigma)|} \leq g(s,q) n^{s-|V(\sigma)|}\leq g(s,q) n^{s-\frac{1}{2}-\sqrt{2j}},
\end{align}
where the second inequality is because $j=|\sigma|\leq {|V(\sigma)|\choose 2}$ implies that $|V(\sigma)|\geq \frac{1+\sqrt{1+8j}}{2}>\frac{1}{2}+\sqrt{2j}$.  Thus we have shown (\ref{dsigma}) holds for all $\sigma\subseteq W$ with $|\sigma|=j$.  Therefore, for all $x\in W$,
\begin{align}\label{ineq00}
d^{(j)}(x)\leq g(s,q)n^{s-\frac{1}{2}-\sqrt{2j}}.
\end{align}
On the other hand, the average degree of $H$ is
\begin{align*}
d=\frac{{s\choose 2}g(s,q){n \choose s}}{ {(q+1){n\choose 2}}} = \frac{g(s,q)}{q+1}{n-2 \choose s-2}\geq \frac{g(s,q)}{q+1}\Big(\frac{n-2}{s-2}\Big)^{s-2}\geq \frac{g(s,q)}{q+1}\Big(\frac{n}{s}\Big)^{s-2} \geq n^{s-2.1},
\end{align*}
where the last two inequalities are because $n$ is sufficiently large.  Combining this with (\ref{ineq00}) yields 
\begin{align}\label{ineq3}
\nonumber\Delta_j=\frac{\sum_{x\in W} d^{(j)}(x)}{dN\tau^{j-1}}\leq \frac{Ng(s,q)n^{s-\frac{1}{2}-\sqrt{2j}}}{dN\tau^{j-1}}=\frac{g(s,q)n^{s-\frac{1}{2}-\sqrt{2j}}}{d\tau^{j-1}}&\leq g(s,q)n^{s-\frac{1}{2}-\sqrt{2j} -s+2.1 +(j-1)\frac{1}{4s}}\\
\nonumber&= g(s,q)n^{1.6 -\sqrt{2j} +(j-1)\frac{1}{4s}}\\
&\leq n^{1.61 -\sqrt{2j} +(j-1)\frac{1}{4s}},
\end{align}
where the last inequality is because $n$ is large.  Since $2\leq j\leq {s\choose 2}$, $\sqrt{2(j-1)}<s$.  Therefore,
$$
\sqrt{2j}-\frac{(j-1)}{4s}\geq \sqrt{2j}-\frac{(j-1)}{4\sqrt{2(j-1)}}= \sqrt{2j}-\frac{\sqrt{j-1}}{4\sqrt{2}}=\sqrt{2j}\Big(1-\frac{\sqrt{j-1}}{8\sqrt{j}}\Big)\geq 2\Big(1-\frac{1}{8}\Big)=1.75.
$$
Combining this with (\ref{ineq3}) we obtain that $\Delta_j\leq n^{1.61 -1.75}=n^{-.14}$. Since this bound holds for all $2\leq j\leq {s\choose 2}$, we have
\begin{align}\label{delta}
\Delta(H,\tau)=&2^{{{s\choose 2}\choose 2}-1}\sum_{j=2}^{s\choose 2} 2^{-{j-1\choose 2}}\Delta_j \leq n^{-.14}\Big(2^{{{s\choose 2}\choose 2}-1}\sum_{j=2}^{s\choose 2} 2^{-{j-1\choose 2}}\Big).
\end{align}
Since $2^{{{s\choose 2}\choose 2}-1}\sum_{j=2}^{s\choose 2} 2^{-{j-1\choose 2}}$ is a constant and $n$ is sufficiently large, (\ref{delta}) implies that we have $\Delta(H,\tau) \leq \epsilon /(12{s\choose 2}!)$, as desired.  We have now verified the hypotheses of Theorem \ref{containers} hold.  Consequently, Theorem \ref{containers} implies there exists  $\mathcal{C}\subseteq 2^W$ such that the following hold.
\begin{enumerate}
\item For every independent set $I$ in $H$, there is some $C\in \mathcal{C}$, such that $I\subseteq C$.
\item For all $C\in \mathcal{C}$, we have $e(H[C])\leq \epsilon e(H)$.
\item $\log |\mathcal{C}| \leq c_1\log(1/\epsilon) N\tau \log (1/\tau)$.
\end{enumerate}
For each $C\in \mathcal{C}$, define $G_C=([n],w^C)$ where for each $f\in {[n]\choose 2}$, $w^C(f) = \max\{ u: (f,u) \in C\}$.  Set $\mathcal{G}=\{G_C: C\in \mathcal{C}\}$.  We show this $\mathcal{G}$ satisfies (i)-(iii) of Theorem \ref{containerscor} for $c$ and $\delta$.  First note that by construction, every $G_C\in \mathcal{G}$ has multiplicity at most $q$.  We now show (i) holds.  Fix $J=([n],w)\in F(n,s,q)$ and let $I=\{(f,w(f)): f\in {[n]\choose 2}\}$.  It is straightforward to verify that because $J$ is an $(s,q)$-graph, $I\subseteq W$ is an independent set in $H$.  By 1, there is $C\in \mathcal{C}$ such that $I\subseteq C$.  By definition of $G_C$, this implies that for each $f\in {[n]\choose 2}$, $w(f)\leq w^{C}(f)$.  In other words, $J$ is a submultigraph of $G_C$.   Thus $\mathcal{G}$ satisfies (i).  

We now show part (ii).  Fix $G_C=([n],w^C)\in \calG$.  By 2, $e(H[C])\leq \epsilon e(H)= \delta {n\choose s}$, where the equality is by definition of $\epsilon$ and because $e(H)=g(s,q){n\choose s}$.  So it suffices to show that $|\calH(G_C,s,q)|\leq e(H[C])$.  Given $A\in \calH(G_C,s,q)$, define $\Theta(A)=\{(f,w^C(f)): f\in {A\choose 2}\}$.  We show $\Theta$ is an injective map from $\calH(G_C,s,q)$ into $E(H[C])$.  By definition of $w^C$, $\Theta(A)\subseteq C$, and by definition of $E$, $\Theta(A)\in E$.  Thus $\Theta(A)\in E(H[C])$.  Clearly for all $A\neq A'\in \calH(G_C,s,q)$, $\Theta(A)\neq \Theta(A')$.  So $\Theta$ is an injection from $\calH(G_C,s,q)\rightarrow E(H[C])$, and $|\calH(G_C,s,q)|\leq e(H[C])$, finishing our proof of (ii). For (iii), we must compute an upper bound for $\log |\mathcal{G}|$.  By definition of $\mathcal{G}$, $|\mathcal{G}|\leq |\mathcal{C}|$, so it suffices to bound $\log |\mathcal{C}|$.  By 3 and the definitions of $N$, $\tau$, $\epsilon$, and $c$, we have
\begin{align*}
\log |\mathcal{C}| \leq c_1\log\Big(\frac{1}{\epsilon}\Big)N\tau \log\Big(\frac{1}{\tau}\Big)=c_1\log\Big(\frac{1}{\epsilon}\Big)(q+1){n\choose 2}n^{-\frac{1}{4s}} \log(n^{\frac{1}{4s}})=c{n\choose 2}n^{-\frac{1}{4s}}\log n.
\end{align*}
Since ${n\choose 2}\leq n^2$, this is at most $cn^{2-\frac{1}{4s}}\log n$.  So $\log|\calG|\leq cn^{2-\frac{1}{4s}}\log n$, as desired.
\end{proof}

We now state a generalization of the graph removal lemma, Lemma \ref{triangleremoval} below.  Since the argument is merely an adjustment of the argument for graphs, using a multi-colored version of Szemer\'edi's regularity lemma (see \cite{AxMa}), we omit the proof.  Given two multigraphs $G=(V,w)$ and $G'=(V,w')$, set $\Delta(G,G')=\big\{xy\in \binom{V}{2}: w(xy)\neq w'(xy)\big\}$.  We say $G$ and $G'$ are \emph{$\delta$-close} if $|\Delta(G,G')|\leq \delta n^2$, otherwise they are \emph{$\delta$-far}.  

\begin{lemma}\label{triangleremoval}
Fix integers $s\geq 2$ and $q\geq 0$.  For all $0<\nu<1$, there is $0<\delta<1$ such that for all sufficiently large $n$, the following holds.  If $G=([n],w)$ satisfies $\mu(G)\leq q$ and $|\calH(G,s,q)|\leq \delta{n\choose 2}$, then $G$ is $\nu$-close to some $G'$ in $F(n,s,q)$.
\end{lemma}

Given $G=(V,w)$, let $G^+=(V,w^+)$ where $w^+(x,y)=w(x,y)+1$ for each $xy\in {V\choose 2}$.  Observe that for any finite multigraph $G$, the number of submultigraphs of $G$ is $P(G^+)$, and if $G\in F(n,s,q)$, then $G^+\in F(n,s,q+{s\choose 2})$.  The following supersaturation type result is a consequence of Lemma \ref{triangleremoval} and Proposition \ref{density}.

\begin{lemma}\label{countinglemma}
Suppose $s\geq 2$, $q\geq 0$. For all $\epsilon>0$ there is a $\delta>0$ such that for all sufficiently large $n$, the following holds.  Suppose $G=([n],w)$ has $\mu(G)\leq q$ and satisfies $|\calH(G,s,q)|\leq \delta{n\choose s}$.  Then $P(G^+)\leq \expi(s,q+{s\choose 2})^{(1+\epsilon){n\choose 2}}$ if $\expi(s,q+{s\choose 2})>1$ and $P(G^+)\leq 2^{\epsilon {n\choose 2}}$ if $\expi(s,q+{s\choose 2})=1$.
\end{lemma}
\begin{proof}
By Proposition \ref{density}, $\expi(s,q+{s\choose 2})$ exists and is at least $1$.  Fix $\epsilon>0$ and set 
\[
B=\begin{cases} \expi(s,q+{s\choose 2})&\text{ if }\expi(s,q)>1\\
2&\text{ if }\expi(s,q)=1.
\end{cases}\]

Set $\nu=\epsilon/(6\log_B(q+1))$.  Apply Lemma \ref{triangleremoval} to $\nu$ to obtain $\delta$.  Assume $n$ is sufficiently large, and $G=([n],w)$ has $\mu(G)\leq q$ and satisfies $|\calH(G,s,q)|\leq \delta{n\choose s}$.  Then Lemma \ref{triangleremoval} implies there is $H=([n],w')\in F(n,s,q)$ such that $G$ and $H$ are $\nu$-close.  Observe $\Delta(G,H)=\Delta(G^+,H^+)$.  Then
\begin{align}\label{ineq7w}
\nonumber P(G^+)=\prod_{xy\in {[n]\choose 2}}(w(xy)+1)&=\Big(\prod_{xy\in {[n]\choose 2}\setminus \Delta(G,H)}(w'(xy)+1)\Big)\Big(\prod_{xy\in \Delta(G,H)}(w(xy)+1)\Big)\\
&=P(H^+)\Big(\prod_{xy\in \Delta(G,H)}\frac{w(xy)+1}{w'(xy)+1}\Big).
\end{align}
Since $\max\{\mu(G),\mu(H)\}\leq q$, we have that for all $xy\in {[n]\choose 2}$, $1\leq w(xy)+1,w'(xy)+1\leq q+1$, and therefore, $\frac{w(xy)+1}{w'(xy)+1}\leq \frac{q+1}{1}=q$.  By assumption, $|\Delta(G,H)|\leq \nu n^2\leq 3\nu{n\choose 2}$.  Observe that because $H^+\in F(n,s,q+{s\choose 2})$, we have $P(H^+)\leq \expi(n,s,q+{s\choose 2})$.  Combining these facts with (\ref{ineq7w}) yields the following. 
\begin{align}\label{ineq4}
P(G^+)\leq P(H^+)(q+1)^{|\Delta(G,H)|}\leq P(H^+)(q+1)^{3\nu {n\choose 2}}\leq \expi\Big(n,s,q+{s\choose 2}\Big)(q+1)^{3\nu {n\choose 2}}.
\end{align}
Because $n$ is sufficiently large an by definition of $\expi(s,q+{s\choose 2})$, we may assume that $\expi(n,s,q+{s\choose 2})$ is at most $\expi(s,q+{s\choose 2})^{n\choose 2}B^{\epsilon{n\choose 2}/2}$.  Combining this with (\ref{ineq4}) and the definition of $\nu$ yields that 
$$
P(G^+)\leq  \expi\Big(s,q+{s\choose 2}\Big)^{{n\choose 2}}B^{\epsilon{n\choose 2}/2}(q+1)^{3\nu {n\choose 2}}=\expi\Big(s,q+{s\choose 2}\Big)^{{n\choose 2}}B^{\epsilon{n\choose 2}}.
$$
When $\expi(s,q+{s\choose 2})>1$, this says $P(G^+)\leq \expi\Big(s,q+{s\choose 2}\Big)^{(1+\epsilon){n\choose 2}}$, and when $\expi(s,q+{s\choose 2})=1$, this says $P(G^+)\leq  2^{\epsilon{n\choose 2}}$.
\end{proof}

\noindent{\bf Proof of Theorem \ref{counting}.} Suppose $s\geq 2$ and $q\geq 0$.  Fix $\epsilon >0$.  We show that for sufficiently large $n$,  
$$
\expi\Big(s,q+{s\choose 2}\Big)^{n\choose 2}\leq |F(n,s,q)|\leq \expi\Big(s,q+{s\choose 2}\Big)^{(1+\epsilon){n\choose 2}}
$$
if $\expi(s,q+{s\choose 2})\ge 1$ and $|F(n,s,q)|\leq 2^{\epsilon {n\choose 2}}$ if $\expi\Big(s,q+{s\choose 2}\Big)=1$.  We first prove the upper bounds.  Set 
\[
B=\begin{cases} \expi(s,q+{s\choose 2})&\text{ if }\expi(s,q)>1\\
2&\text{ if }\expi(s,q)=1.
\end{cases}\]
Apply Lemma \ref{countinglemma} to $\epsilon /2$ to obtain $\delta$ and apply Theorem \ref{containerscor} to $\delta$ for $s$ and $q$ to obtain $c$.  Assume $n$ is sufficiently large.  By Theorem \ref{containerscor}, there is a collection $\mathcal{G}$ of multigraphs of multiplicity at most $q$ and with vertex set $[n]$ such that 
\begin{enumerate}[(i)]
\item for every $J\in F(n,s,q)$, there is $G\in \mathcal{G}$ such that $J$ is a full submultigraph of $G$,
\item for every $G\in \mathcal{G}$, $|\calH(G,s,q)|\leq \delta {n\choose s}$, and 
\item $\log |\mathcal{G}| \leq cn^{2-\frac{1}{4s}}\log n$.
\end{enumerate}
By Lemma \ref{countinglemma} and (ii), for every $G\in \calG$, $P(G^+)\leq \expi(s,q+{s\choose 2})^{{n\choose 2}}B^{\epsilon{n\choose 2}/2}$.  By (i), every element of $F(n,s,q)$ can be constructed as follows.
\begin{enumerate}[$\bullet$]
\item Choose $G\in \mathcal{G}$.  By (iii), there are at most $cn^{2-\frac{1}{2s}}\log n$ choices.  Since $n$ is sufficiently large, we may assume that $cn^{2-\frac{1}{4s}}\log (n) \leq B^{\epsilon{n\choose 2}/2}$.
\item Choose a submultigraph of $G$.  There are $P(G^+)\leq \expi(s,q+{s\choose 2})^{{n\choose 2}}B^{\epsilon{n\choose 2}/2}$ choices.
\end{enumerate}
Combining these bounds yields that $|F(n,s,q)|\leq \expi(s,q+{s\choose 2})^{{n\choose 2}}B^{\epsilon{n\choose 2}}$.  In other words, if $\expi(s,q+{s\choose 2})>1$, then $|F(n,s,q)|\leq \expi(s,q+{s\choose 2})^{(1+\epsilon){n\choose 2}}$ and if $\expi(s,q+{s\choose 2})=1$, then $|F(n,s,q)|\leq 2^{\epsilon {n\choose 2}}$.  We only have left to show that in the case where $\expi(s, q+{s\choose 2})>1$, $|F(n,s,q)|\geq  \expi(s, q+{s\choose 2})^{{n\choose 2}}$.  Choose any product-extremal $G_0=([n],w_0)\in F(n,s,q+{s\choose 2})$.  Observe that by assumption and Proposition \ref{density}, $P(G_0)=\expi(n,s,q+{s\choose 2})\geq \expi(s,q+{s\choose 2})^{n\choose 2}>1$.  Therefore, $G_0$ contains no edges of multiplicity $0$, so we can define a multigraph $G=([n],w)$ satisfying $w(xy)=w_0(xy)-1$ for all $xy\in {[n]\choose 2}$.  By construction, $G^+=G_0$, so $G_0\in F(n,s,q+{s\choose 2})$ implies $G\in F(n,s,q)$.   Then $|F(n,s,q)|$ is at least the number of submultigraphs of $G$, which is $P(G^+)=P(G_0)=\expi(n,s,q+{s\choose 2})\geq \expi(s,q+{s\choose 2})^{n\choose 2}$. 
\qed

\section{Extremal result for $(n,4,15)$-graphs: a two-step reduction}\label{extcaseivsection}


In this section we reduce Theorem \ref{caseiv} to two stepping-stone theorems, Theorems \ref{EXTGAREIND} and \ref{EXTGINDAREINW}, below.   The main idea is that Theorem \ref{caseiv} relies on understanding the structure of $(4,15)$-graphs which are product-extremal subject to certain constraints. Given a set $\calF$ of multigraphs, recall that
$$
\calP(\calF)=\{ G\in \calF: P(G)\geq P(G')\text{ for all }G'\in \calF\} \qquad \hbox{ and}\qquad \calP(n,4,15)=\calP(F(n,4,15)).
$$

\begin{definition}
Given $n\in \mathbb{N}$, define $F_{\leq 3}(n,4,15)=\{G\in F(n,4,15): \mu(G)\leq 3\}$ and 
$$
D(n)=F_{\leq 3}(n,4,15)\cap F(n,3,8).
$$
\end{definition}

\begin{theorem}\label{EXTGAREIND}
For all sufficiently large $n$, $\calP(n,4,15)=\calP(D(n))$.
\end{theorem}

\begin{theorem}\label{EXTGINDAREINW}
For all sufficiently large $n$, $\calP(D(n))\cap \calP(W(n))\neq \emptyset$.
\end{theorem}

These two theorems will be proved in Sections \ref{sectionEXTGAREIND} and \ref{sectionEXTGINDAREINW} respectively.  We use the rest of this section to prove Theorem \ref{caseiv}, given Theorems \ref{EXTGAREIND} and \ref{EXTGINDAREINW}.  Given $G=([n],w)\in W(n)$, let $L(G)$ and $R(G)$ denote the parts in the partition of $[n]$ such that $w(xy)=1$ if and only if $xy\in {L(G)\choose 2}$. Recall the definition of $\gamma$ from Theorem \ref{caseiv}.

\begin{lemma}\label{extW}
For all $G\in \calP(W(n))$, we have $P(G)=2^{\gamma n^2+O(n)}$.
\end{lemma}

\begin{proof}
Let $G=([n],w)\in W(n)$.  Set $h(y)=2^{y\choose 2}3^{y(n-y)}$ and observe that if $|L(G)|=n-y$ and $|R(G)|=y$, then $P(G)=h(y)$.  Thus it suffices to show that $\max_{y\in [n]}h(y)=2^{\gamma n^2+O(n)}$.  Basic calculus shows that $h(y)$ has a global maximum at $\tau=\beta n - (\log 2)/(2(2\log 3-\log2))$, where $\beta= \frac{\log 3}{2\log 3-\log 2}$ is as in Theorem \ref{caseiv}.  This implies $\max_{y\in \mathbb{N}}h(y) = \max\{ h(\lfloor \tau \rfloor), h(\lceil \tau \rceil)\}$.  It is straightforward to check  $\max\{ h(\lfloor \tau \rfloor), h(\lceil \tau \rceil)\}=\max\{ h(\lfloor \beta n, \rfloor ), h(\lceil \beta n \rceil) \}$.   By definition of $\gamma$ and $h$, this implies $\max_{y\in [n]}h(y)=2^{\gamma n^2+O(n)}$.
\end{proof}

\vspace{2mm}

\noindent {\bf Proof of Theorem \ref{caseiv}.}
Fix $n$ sufficiently large and $G_1\in \calP(W(n))$.  By Theorem \ref{EXTGINDAREINW}, there is some $G_2\in \calP(D(n))\cap \calP(W(n))$.  Since $G_1$ and $G_2$ are both in $\calP(W(n))$, $P(G_1)=P(G_2)$.  Our assumption and Theorem \ref{EXTGAREIND} imply $G_2\in \calP(D(n))=\calP(n,4,15)$, so $P(G_2)=\expi(n,4,15)$.  Combining these facts yields $P(G_1)=P(G_2)=\expi(n,4,15)$, so $G_1\in \calP(n,4,15)$.  This shows $\calP(W(n))\subseteq \calP(n,4,15)$.  Since $G_1\in \calP(n,4,15)\cap\calP(W(n))$, Lemma \ref{extW} implies $\expi(n,4,15)=P(G_1)=2^{\gamma n^2+O(n)}$.  By definition, $\expi(4,15)=2^{2\gamma}$.
\qed

\vspace{2mm}

\section{Proof of Theorem \ref{EXTGINDAREINW}}\label{sectionEXTGINDAREINW}

The goal of this section is to prove Theorem \ref{EXTGINDAREINW}.  It will require many reductions and lemmas.  The general strategy is to show we can find elements in $\calP(D(n))$ with increasingly nice properties, until we can show there is one in $W(n)$.  The proof methods can be viewed as a generalization of Zykov-symmetrization to multigraphs, where we successively replace and duplicate vertices if they do not have certain desirable properties. 

\subsection{Finding an element of $\calP(D(n))$ in $C(n)$}

Given $G=(V,w)$ and $i,j,k\in \mathbb{N}$, an \emph{$(i,j,k)$-triangle} in $G$ is a set $\{x,y,z\} \in {V\choose 3}$ such that $\{w(xy), w(yz), w(xz)\}=\{i,j,k\}$.  Say that $G$ \emph{omits $(i,j,k)$-triangles} if there is no $(i,j,k)$-triangle in $G$. 

\begin{definition}
Suppose $n\geq 1$.  Define $A_{i,j,k}(n)=\{G\in F(n,4,15): G\text{ omits }(i,j,k)\text{-triangles}\}$ for each $i,j,k\in \mathbb{N}$, and set
\begin{align*}
 C(n)=D(n)\cap A_{3,1,1}(n)\cap A_{2,1,1}(n)\cap A_{3,2,1}(n).
\end{align*}
\end{definition}
Observe that for all $n$, $W(n)\subseteq C(n)\subseteq D(n)$.  The goal of this subsection is to prove Lemma \ref{EXTGINBCAPAAREINC}, which says the for all $n$, there is a product-extremal element of $D(n)$ which is also in $C(n)$.  We begin with some notation.  Suppose $G=(V,w)$ and $x\neq y\in V$. Define $G_{xy}=(V,w')$ to be the multigraph such that 
\begin{itemize}
\item $G_{xy}[V\setminus \{x,y\}]=G[V\setminus \{x,y\}]$,
\item $w'(xy)=1$, and 
\item for all $u\in V\setminus \{x,y\}$, $w'(xu)=w(yu)$.
\end{itemize}
The idea is that $G_{xy}$ is obtained from $G$ by making the vertex $x$ ``look like'' the vertex $y$.  Given $xy, vu\in {V\choose 2}$, define
$$
G_{vu,xy} = (G_{uv})_{vu}.
$$
Given $G=(V,w)$ and $y\in V$, set $p(y)=\prod_{x\in V\setminus \{y\}}w(xy)$.  We will use the following two equations for any $xy\in {V\choose 2}$ and $\{u,v,z\}\in {V\choose 3}$.  
\begin{align}
P(G_{xy})&=\frac{p(y)}{p(x)w(xy)}P(G)\text{ and }\label{line1}\\
P(G_{vu,zu})&=\frac{p(u)^2w(vz)}{p(v)p(z)w(uz)^2w(uv)^2}P(G)\label{line3*}.
\end{align}

\begin{lemma}\label{replacement}
Suppose $n\geq 1$, $G\in D(n)$, and $uv$, $xy\in {[n]\choose 2}$. Then $G_{uv}$ and $G_{uv,xy}$ are both in $D(n)$.
\end{lemma}
\begin{proof}
Fix $G=([n],w)\in D(n)$ and let $G':=G_{uv}=([n],w')$.  We show $G' \in D(n)$.  Given $X\subseteq [n]$, let $S(X)=\sum_{xy\in {X\choose 2}}w(xy)$ and $S'(X)=\sum_{xy\in {X\choose 2}}w'(xy)$.  By definition of $G_{uv}$ and because $G\in D(n)$, $\mu(G')\leq 3$.  We now check that $G'\in F(n,4,15)$.  Suppose $X\in {[n]\choose 4}$.  If $u\notin X$, then $S'(X)=S(X)\leq 15$.  If $X\cap \{u,v\}=\{u\}$, then $S'(X)=S((X\setminus \{u\})\cup \{v\})\leq 15$. So assume $\{u,v\}\subseteq X$, say $X=\{u,v,z,z'\}$.  Because $G\in F(n,3,8)$ and by definition of $G_{uv}$, we have that $S'(\{v,z,z'\})=S(\{v,z,z'\})\leq 8$.  Combining this with the facts that $w'(uv)=1$ and $\mu(G')\leq 3$ yields
$$
S'(X)=S'(\{v,x,y\})+w'(uv)+w'(ux)+w'(uy)\leq 8+1+3+3=15.
$$
We now verify that $G'\in F(n,3,8)$.   Suppose $X\in {[n]\choose 3}$.  If $u\notin X$, then $S'(X)=S(X)\leq 8$.  If $X\cap \{u,v\}=\{u\}$, then $S'(X)=S((X\setminus \{u\})\cup \{v\})\leq 8$.  So assume $\{u,v\}\subseteq X$, say $X=\{u,v,z\}$.  Because $\mu(G')\leq 3$, 
$$
S'(X)\leq w'(uv)+3+3=1+3+3=7\leq 8.
$$
Consequently, $G'\in F_{\leq 3}(n,4,15)\cap F(n,3,8)=D(n)$.  Repeating the proof yields $(G')_{xy}\in D(n)$.
\end{proof}

\begin{lemma}\label{triangles1}
For all $n\geq 1$, if $G\in \calP(D(n))$, then $G$ contains no $(3,1,1)$-triangle or $(2,1,1)$-triangle.
\end{lemma}
\begin{proof}
Suppose towards a contradiction that $G=([n],w)\in \calP(D(n))$ and $\{u,v,z\}\in {[n]\choose 3}$ is a $(3,1,1)$-triangle or a $(2,1,1)$-triangle.   Assume $w(uv)=w(uz)=1$ and $w(vz)\in\{2,3\}$. Without loss of generality assume $p(v)\geq p(z)$.  Note that by Lemma \ref{replacement}, $G_{uv}$ and $G_{vu,zu}$ are in $D(n)$.  If $p(v)>p(u)$, then using (\ref{line1}) and $w(uv)=1$ we obtain
\begin{align*}
P(G_{uv})&=\frac{p(v)}{p(u)} P(G)>P(G),
\end{align*}
which implies $G\notin \calP(D(n))$.  Therefore we may assume $p(z)\leq p(v)\leq p(u)$. Using  (\ref{line3*}) and $w(vz)\geq 2$, we obtain 
$$
P(G_{vu,zu})=\frac{w(vz)p(u)^2}{p(v)p(z)}P(G)\geq w(vz)P(G)\geq 2P(G)>P(G),
$$
a contradiction. 
\end{proof}

\noindent Given $G\in F(n,4,15)$, set $\Gamma(G)=\{Y\in {[n]\choose 3}: Y$ is a $(1,2,3)$-triangle in $G\}$.

\begin{lemma}\label{123lemma}
Suppose $n\geq 1$, $G=([n],w)\in D(n)$, and $u,v,z\in [n]$ are such that $w(uv)=1$, $w(uz)=2$ and $w(vz)=3$.  Then either $|\Gamma(G_{uv})|<|\Gamma(G)|$ or $|\Gamma(G_{vu})|<|\Gamma(G)|$.
\end{lemma}
\begin{proof}
Let $X=\{u,v,z\}$.  Given $y,y'\in X$, set
$$
\Gamma_y= \{x, x'\in [n]\setminus X: \{y,x,x'\} \in \Gamma(G)\}\quad \hbox{ and }\quad \Gamma_{yy'}= \{x\in [n]\setminus X: \{y,y',x\} \in \Gamma(G)\}.
$$
Observe that
$$
\Gamma(G)=\Gamma(G[[n]\setminus X])\cup\Gamma_u\cup \Gamma_v\cup \Gamma_z\cup \Gamma_{uv}\cup\Gamma_{vz}\cup \Gamma_{uz}\cup \{ X\},
$$
so $|\Gamma(G)|=|\Gamma(G[[n]\setminus X])|+|\Gamma_u|+|\Gamma_v|+|\Gamma_z|+|\Gamma_{uv}|+|\Gamma_{vz}|+|\Gamma_{uz}|+1$.  Let $G_{uv}=([n],w^{G_{uv}})$ and $G_{vu}=([n],w^{G_{vu}})$.  Note that for all $x\in [n]\setminus \{u,v\}$, we have $w^{G_{uv}}(vx)=w^{G_{uv}}(ux)$ and $w^{G_{vu}}(vx)=w^{G_{vu}}(ux)$, so there are no $(1,2,3)$-triangles in $G_{uv}$ or $G_{vu}$ of the form $\{u,v, x\}$.  If $x\in [n]\setminus X$ is such that $\{x,v,z\}\in \Gamma(G)$, then $\{x,v,z\},\{x,u,z\}\in \Gamma(G_{uv})$.  Similarly, if $x,y\in [n]\setminus X$ are such that $\{x,y,v\}\in \Gamma(G)$, then $\{x,y,v\},\{x,y,u\}\in \Gamma(G_{uv})$.  Combining these observations, we have that $|\Gamma(G_{uv})|=|\Gamma(G[[n]\setminus X])|+|\Gamma_z|+2|\Gamma_v|+2|\Gamma_{vz}|$. The same argument with the roles of $u$ and $v$ switched implies $|\Gamma(G_{vu})|=|\Gamma(G[[n]\setminus X])|+|\Gamma_z|+2|\Gamma_u|+2|\Gamma_{uz}|$.  Suppose first that $|\Gamma_v|+|\Gamma_{vz}|\leq |\Gamma_u|+|\Gamma_{uz}|$.  Then 
$$
|\Gamma(G_{uv})|\leq |\Gamma(G[[n]\setminus X])|+|\Gamma_z|+|\Gamma_u|+|\Gamma_v|+|\Gamma_{uz}|+|\Gamma_{vz}|\leq |\Gamma(G)|-1.
$$
If on the other hand, $|\Gamma_v|+|\Gamma_{vz}|\geq |\Gamma_u|+|\Gamma_{uz}|$, then the same argument with the roles of $u$ and $v$ switched implies  $|\Gamma(G_{vu})|\leq |\Gamma(G)|-1$.
\end{proof}

\begin{lemma}\label{triangles2}
For any $n\geq 1$ and $G\in D(n)$, there is $H\in D(n)\cap A_{1,2,3}(n)$ such that $P(H)\geq P(G)$.
\end{lemma}
\begin{proof}
Suppose $G\in D(n)$ satisfies $\Gamma(G)\neq \emptyset$.  We give a procedure for defining $H(G)\in D(n)$ such that either $P(H(G))>P(G)$ or $P(H(G))=P(G)$ and $|\Gamma(H(G))|<|\Gamma(G)|$.  Choose some $\{u,v,z\}\in \Gamma(G)$, say $w(uv)=1$, $w(uz)=2$, and $w(vz)=3$. Suppose $p(v)<p(u)$.  Then Lemma \ref{replacement} implies $G_{uv}\in D(n)$, and (\ref{line1}) along with $w(uv)=1$ imply $P(G_{uv})= (p(v)/p(u))P(G)>P(G)$, so set $H(G)=G_{uv}$.  If $p(u)<p(v)$, the same argument with the roles of $u$ and $v$ switched implies $G_{vu}\in D(n)$ and $P(G_{vu})>P(G)$, so set $H(G)=G_{vu}$.  If $p(u)=p(v)$, use Lemma \ref{123lemma} to choose $H(G)=G_{uv}$ or $H(G)=G_{vu}$ such that $|\Gamma(H(G))|<|\Gamma(G)|$.  In this case, $P(G)=P(H(G))$.

Now fix $G\in D(n)$.  Define a sequence $G_1,\ldots, G_k$ as follows.  Set $G_1=G$.  Suppose $i>1$ and $G_1,\ldots, G_i$ have been defined.  If $\Gamma(G_i)=\emptyset$, set $k=i$.  If $\Gamma(G_i)\neq \emptyset$, set $G_{i+1}=H(G_i)$.  Clearly this algorithm will end after at some finite number of steps.  The resulting $G_k$ will contain no $(1,2,3)$-triangles and will satisfy $P(G_k)\geq P(G)$.
\end{proof}

We now prove the main result of this subsection. 
\begin{lemma}\label{EXTGINBCAPAAREINC}
For all $n\geq 1$, $\calP(D(n))\cap C(n)\neq \emptyset$.  Consequently, $\calP(C(n))\subseteq \calP(D(n))$.
\end{lemma}
\begin{proof}
Suppose $G\in \calP(D(n))$.  Lemma \ref{triangles2} implies there is $H\in D(n)\cap A_{1,2,3}(n)$ with $P(H)\geq P(G)$.  Since $G\in \calP(D(n))$, this implies $H\in \calP(D(n))$.   Lemma \ref{triangles1} implies $H\in A_{2,1,1}(n)\cap A_{3,1,1}(n)$.  Therefore $H\in C(n)$.  This shows $\calP(D(n))\cap C(n)\neq \emptyset$.  Combining this with $C(n)\subseteq D(n)$ yields that $\calP(C(n))\subseteq \calP(D(n))$.
\end{proof}

\subsection{Acyclic multigraphs}

We say two multigraphs $G=(V,w)$ and $G'=(V',w)$ are $\emph{isomorphic}$, denoted $G\cong G'$, if there is a bijection $f: V\rightarrow V'$ such that $w(xy)=w'(f(x)f(y))$, for all $xy\in {V\choose 2}$.  We say that $G=(V,w)$ \emph{contains a copy of $G'$} if there is $X\subseteq V$ such that $G[X]\cong G'$.

\begin{definition}
Given $t\geq 3$, define $C_t(3,2)$ to be the multigraph $([t],w)$ such that
$$
w(12)=w(23)=\ldots =w((t-1)t)=w(t1)=3,
$$
and $w(ij)=2$ for all other pairs $i\neq j$.  For $n\geq 1$, set $NC(n)$ (NC=``no cycles'') to be the set of $G\in C(n)$ which do not contain a copy of $C_t(3,2)$ for any $t\geq 3$.
\end{definition}
We will show in the next subsection that for large $n$, all product-extremal elements of $C(n)$ are in $NC(n)$.  However, we must first show that we can find product-extremal elements of $NC(n)$ which are ``nice,'' and this is the goal of this subsection.  In particular we will show that for all $n\geq 1$, there is a product-extremal element of $NC(n)$ which is also in $W(n)$.

We begin with some notation and definitions.   If $G$ contains a copy of $C_t(3,2)$, we will write $C_t(3,2)\subseteq G$, and if not, we will write $C_t(3,2)\nsubseteq G$.  A \emph{vertex-weighted graph} is a triple $(V,E,f)$ where $(V,E)$ is graph and $f:V\rightarrow \mathbb{N}^{>0}$.  Given a multigraph $G=(V,w)$, let $\sim_G$ be the binary relation on $V$ defined by $x\sim_G y\Leftrightarrow w(xy)=1$.  

\begin{definition}
A multigraph $G$ is \emph{neat} if $\mu(G)\leq 3$ and $G$ contains no $(i,j,k)$-triangle for $(i,j,k)\in \{(1,1,2), (1,1,3), (1,2,3)\}$.  
\end{definition}

\noindent Observe that all multigraphs in $C(n)$ are neat.  Neat multigraphs have the property that we can ``mod out'' by $\sim_G$ in a coherent way.

\begin{proposition}\label{caseivprop}
Suppose $G=(V,w)$ is a neat multigraph.  Then $\sim_G$ forms an equivalence relation on $V$.  Moreover, if $\tilde{V}=\{V_1,\ldots, V_t\}$ is the set of equivalence classes of $V$ under $\sim_G$, then for each $i\neq j$, there is $w_{ij}\in \{2,3\}$ such that for all $(x,y)\in V_i\times V_j$, $w(xy)=w_{ij}$.  
\end{proposition}

\noindent The proof is straightforward and left to the reader. Suppose $G=(V,w)$ is a neat multigraph, $\tilde{V}=\{V_1,\ldots, V_t\}$ is the set of equivalence classes of $V$ under $\sim_G$, and for each $i\neq j$, $w_{ij}\in \{2,3\}$ is from Proposition \ref{caseivprop}.  Define the \emph{vertex-weighted graph associated to $G$ and $\sim_G$} to be $\tilde{G}=(\tilde{V},\tilde{E},f)$ where $\tilde{E}=\{V_iV_j\in {\tilde{V}\choose 2}: w_{ij}=3\}$ and $f(V_i)=|V_i|$ for all $i\in [t]$.  We will use the notation $|\cdot|^G$ to denote this vertex-weight function $f$, and we will drop the superscript when $G$ is clear from context.  If $H=(V,E)$ is a graph and $X\subseteq V$, then let $H[X]=(X, E\cap {X\choose 2})$.

\begin{lemma}\label{subgraph}
Suppose $n\geq 1$ and $G$ is a neat multigraph with vertex set $[n]$.  Then $G\in NC(n)$ if and only if $\tilde{G}$ is a forest.
\end{lemma}
\begin{proof}
Suppose $\tildeG$ is not a forest.  Then there is $X=\{V_{i_1},\ldots, V_{i_k}\}\subseteq \tildeV$ such that $\tildeG[X]$ is a cycle of length $k\geq 3$.  Choose some $y_j\in V_{i_j}$ for each $1\leq i\leq k$ and let $Y=\{y_1,\ldots, y_k\}$.  Then  by definition of $\tilde{G}$, we must have $G[Y]\cong C_k(3,2)$.  Thus $G\notin NC(n)$.

On the other hand, suppose $G\notin NC(n)$.  Then because $G$ is neat, we must have that either $G\notin F(n,4,15)$ or $C_t(3,2)\subseteq G$ for some $t\geq 3$.  Suppose $G\notin F(n,4,15)$.   Then there is some $Y\in {[n]\choose 4}$ such that $S^G(Y)>15$.  Since $\mu(G)\leq 3$, this implies that either 
\begin{enumerate}[(i)]
\item $\{w(xy): xy\in {Y\choose 2}\}=\{3,3,3,3,2,2\}$ or 
\item $\{w(xy): xy\in {Y\choose 2}\}=\{3,3,3,3,3,j\}$, some $j\in \{1,2,3\}$.
\end{enumerate}
Let $X$ be the set of equivalence classes intersecting $Y$, that is $X=\{V_i\in \tilde{V}: Y\cap V_i\neq \emptyset\}$.  In Case (i), because $Y$ spans no edges of multiplicity $1$ in $G$, the elements of $Y$ must be in pairwise distinct equivalence classes under $\sim_G$.  Thus in $\tilde{G}$, $|X|=4$ and $X$ spans exactly $4$ edges.  This implies $\tildeG[X]$ is either a $4$-cycle or contains a triangle.  In Case (ii), if $j=1$, then $|X|=3$ and $\tildeG[X]$ is a triangle.  If $j\neq 1$, then $|X|=4$ and spans at least $5$ edges.  This implies $\tildeG[X]$ contains a triangle.  Therefore, if $G\notin F(n,4,15)$, then $\tilde{G}$ is not a forest.  Suppose now $C_t(3,2)\subseteq G$, for some $t\geq 3$.  Then if $X\subseteq [n]$ is such that $G[X]\cong C_t(3,2)$, $\tilde{G}[X]$ is a cycle, so consequently $\tilde{G}$ is not a forest.
\end{proof}

\begin{definition}\label{vwdef}
Given a vertex-weighted graph $\tildeG=(\tildeV,E,|\cdot|)$, set
\begin{align*}
f_{\pi}(\tildeG)&=\prod_{UV\in E}3^{|U||V|}\prod_{UV\in {\tildeV \choose 2}\setminus E}2^{|U||V|}.
\end{align*}
Note that we have $P(G)=f_{\pi}(\tildeG)$ for all $G\in C(n)$.
\end{definition}

\noindent Two vertex-weighted graphs $G_1=(V_1,E_1,f_1)$ and $G_2=(V_2,E_2,f_2)$, are \emph{isomorphic}, denoted $G_1\cong G_2$, if there is a graph isomorphism $g:V_1\rightarrow V_2$ such that for all $v\in V_1$, $f_1(v)=f_2(g(v))$. 

\begin{lemma}\label{vertexwhtlem}  
Suppose $n\geq 1$ and $H=(\tilde{V},E,|\cdot|)$ is a vertex-weighted forest such that $\sum_{V\in \tilde{V}}|V|=n$.  Then there is a multigraph $G\in NC(n)$ such that $\tilde{G}$ is isomorphic to $H$.
\end{lemma}
\begin{proof}
Let $\tilde{V}=\{V_1,\ldots, V_t\}$ and for each $i$, let $x_i=|V_i|$.  Since $\sum_{i=1}^tx_i=n$, it is clear there exists a partition $P_1,\ldots, P_t$ of $[n]$ such that for each $i\in [t]$, $|P_i|=x_i$.  Fix such a partition $P_1,\ldots, P_t$.  Define $G=([n],w)$ as follows.  For each $xy\in {[n]\choose 2}$, set
\[
w(xy)= \begin{cases}
1 & \text{ if }xy\in {P_i\choose 2}\text{ for some }i\in [t]\\
3 & \text{ if }xy\in E(P_i,P_j) \text{ for some $i\neq j$ such that }V_iV_j\in E\\
2 & \text{ if }xy\in E(P_i,P_j) \text{ for some $i\neq j$ such that }V_iV_j\notin E.
\end{cases}
\]
By construction, $G$ is a neat multigraph and $\tilde{G}$ is isomorphic to $H$.  Because $H\cong \tildeG$ is a forest, Lemma \ref{subgraph} implies $G\in NC(n)$.
\end{proof}

Given a vertex-weighted graph, $H=(\tildeV,E,|\cdot|)$ and $V\in \tildeV$, let $d^H(V)$ to denote the degree of $V$ in the graph $(\tildeV,E)$.  Given a graph $(\tildeV,E)$ and disjoint subsets $\tildeX,\tildeY$ of $\tildeV$, let $E(\tildeX)=E\cap {\tildeX\choose 2}$ and $E(\tildeX,\tildeY)=E\cap \{XY: X\in \tildeX, Y\in \tildeY\}$.

\begin{lemma}\label{starlemma}
Suppose $H=(\tildeV,E,|\cdot|)$ is a vertex-weighted forest such that $(\tildeV,E)$ is not a star.  Then there is a vertex-weighted graph $H'=(\tildeV,E', |\cdot|)$ such that $(\tildeV,E')$ is a star, and
\begin{align*}
f_{\pi}(H')\geq f_{\pi}(H).
\end{align*}
Moreover, if $f_{\pi}(H')=f_{\pi}(H)$, then $|V|=|W|$ where $V$ is the center of the star $(\tildeV,E')$ and $W\in \tildeV$ is some vertex distinct from $V$.
\end{lemma}
\begin{proof}
Let $H=(\tildeV,E,|\cdot|)$ be a vertex-weighted forest.  Fix $V\in \tildeV$ with $|V|=\max\{|X|: X\in \tildeV\}$.  We now define a sequence $H_0, H_1,\ldots, H_k$, where for each $i$, $H_i=(\tildeV, E_i, |\cdot|)$.

Step $0$: Let $\tildeX$ be the set of isolated points in $H$.  If $\tildeX=\emptyset$ set $H_0=H$ and go to the next step.  If $\tildeX\neq \emptyset$, let $E_0=E\cup \{VX: X\in \tildeX\}$ and $H_0=(\tildeV,E_0, |\cdot|)$.  Clearly $(\tildeV,E_0)$ is still a forest, since any cycle must contain a new edge, i.e. an edge of the form $VX$, some $X\in \tildeX$.  But $d^{H_0}(X)=1$ for all $X\in \tildeX$ implies no $X\in \tildeX$ can be contained in a cycle in $H_0$.  Further, note
\begin{align*}
f_{\pi}(H_0)&=f_{\pi}(H)\Big(\frac{3}{2}\Big)^{\sum_{X\in \tildeX}|V||X|}>f_{\pi}(H).
\end{align*}
If $H_0$ is a star, end the construction and let $k=0$, otherwise go to the next step.  

Step $i+1$: Suppose by induction we have defined $H_0, \ldots, H_i$ such that $(\tildeV,E_i)$ is forest but not a star and contains no isolated points.  Since $(\tildeV,E_i)$ is not a star, it is in particular, not a star with center $V$.  This implies the set $\tildeY_{i}:=\tildeV\setminus (\{V\}\cup d^{H_i}(V)) \neq \emptyset$.  We show there is $Y\in \tildeY_i$ such that $d^{H_i}(Y)=1$.  Since there are no isolated points in $(\tildeV,E_i)$, every $Y\in \tildeY_i$ has $d^{H_i}(Y)\geq 1$.  Suppose towards a contradiction that every $Y\in \tildeY_i$ had $d^{H_i}(Y)\geq2$.  Choose a maximal sequence of points $\Ybar=(Y_1,\ldots, Y_u)$ from $\tildeY_i$ with the property that $Y_1Y_2, \ldots, Y_{u-1}Y_u\in E_i$. Since $Y_1$ and $Y_u$ have degree at least two in $(\tildeV,E_i)$ and because $(\tildeV,E_i)$ is a forest, there are $Z_1, Z_u\in \tildeV\setminus \Ybar$ such that $Y_1Z_1, Y_uZ_u\in E_i$.  Since $Y_1,Y_u\in Y_i$, $Z_1,Z_u\neq V$ and since $\Ybar$ was maximal, $Z_1,Z_u\notin \tildeY_i$.  Thus $Z_1,Z_u\in \tildeV\setminus (\tildeY_i\cup \{V\})$ which implies $VZ_1, VZ_u\in E_i$.  This yields that   $V, Z_1, Y_1, \ldots, Y_u, Z_u, V$ is a cycle in $(\tildeV,E_i)$, a contradiction.  Thus there exists $Y\in \tildeY_i$ such that $d^{H_i}(Y)=1$.  Fix such a $Y\in \tildeY_i$ and let $W$ be the unique neighbor of $Y$ in $(\tildeV,E_i)$.  Define 
$$
E_{i+1}=(E_i\setminus \{YW\})\cup \{VY\}.
$$
and let $H_{i+1}=(\tildeV,E_{i+1},|\cdot|)$.  We first check $(\tildeV,E_{i+1})$ is a forest.  Since $(\tildeV,E_i)$ is a forest, any cycle in $(\tildeV,E_{i+1})$ will contain $VY$.  However, $d^{H_{i+1}}(Y)=1$, so $Y$ cannot be contained in a cycle.  Note
$$
f_{\pi}(H_{i+1})=f_{\pi}(H_i)3^{|V||Y|-|Y||W|}2^{|Y||W|-|V||Y|}=f_{\pi}(H_i)\Big(\frac{3}{2}\Big)^{|Y|(|V|-|W|)}\geq f_{\pi}(H_i),
$$
where the inequality holds because $|V|\geq |W|$ by choice of $V$. Further, note that the inequality is strict unless $|V|=|W|$.

Clearly this process must end after some $0\leq k<|\tildeV|$ steps.  If $k=0$, then $H_0=H_k$ is a star and $f_{\pi}(H_k)>f_{\pi}(H)$.  If $k\geq 1$, then the resulting $H_k=(\tildeV,E_k,|\cdot|)$ will have the property that $(\tildeV,E_k)$ is a star with center $V$.  Since $k\geq 1$, one of the following holds.
\begin{itemize}
\item $f_{\pi}(H_1)>f_{\pi}(H_0)$, so $f_{\pi}(H_k)>f_{\pi}(H)$, or 
\item $f_{\pi}(H_0)=f_{\pi}(H_1)$ and at step $1$, we found a vertex $W\neq V$ with $|V|=|W|$.
\end{itemize}
\end{proof}

\begin{lemma}\label{wlemma}
Suppose $n\geq 1$, $G\in NC(n)$, and $\tilde{G}=(\tilde{V}, E, |\cdot|)$ is the vertex-weighted graph associated to $G$ and $\sim_G$.  Suppose $(\tilde{V},E)$ is a star with center $V$ and there is $W\in \tilde{V}\setminus \{V\}$ such that $|W|>1$.  Then $G\notin \calP(NC(n))$.
\end{lemma}

\begin{proof}
Let $\tilde{V}'=(\tilde{V}\setminus \{W\})\cup \{W_1,W_2\}$ and $E'=(E\setminus \{VW\})\cup \{VW_1,VW_2\}$, where $W_1, W_2$ are new vertices.  Let $H=(\tilde{V}',E', |\cdot|')$ where the vertex-weight function $|\cdot|'$ is defined by $|U|'=|U|$ for all $U\in \tilde{V}\setminus \{W\}$, $|W_1|'=|W|-1$, and $|W_2|'=1$.  By definition of $H$, $\sum_{U\in \tilde{V}'}|U|'=\sum_{U\in \tilde{V}} |U| =n$.  Since $H$ is obtained from $\tildeG$ by splitting the degree one vertex $W$ into $W_1$ and $W_2$, and $\tildeG$ is a forest, $H$ is also a forest.  Thus $H$ satisfies the hypotheses of Lemma \ref{vertexwhtlem}, so there is an $G'\in NC(n)$ such that $\tildeG'$ is isomorphic to $H$.  This and Definition \ref{vwdef} implies $f_{\pi}(H)=f_{\pi}(\tildeG')=P(G')$.  Let $\tildeZ=\tildeV\setminus \{V,W\}$.  Then 
\begin{align*}
f_{\pi}(H)&=\Big(\prod_{U\in \tildeZ}3^{|U||V|}\prod_{UU'\in {\tildeZ\choose 2}}2^{|U||U'|}\Big)\Big(\prod_{U\in \tildeZ}2^{|W_1||U|+|W_2||U|}\Big)3^{|V||W_1|'+|V||W_2|'}2^{|W_1|'|W_2|'}\\
&=\Big(\prod_{U\in \tildeZ}3^{|U||V|}\prod_{UU'\in {\tildeZ\choose 2}}2^{|U||U'|}\Big)\Big(\prod_{U\in \tildeZ}2^{|W||U|}\Big)3^{|V||W|}2^{|W|-1}\\
&=f_{\pi}(\tilde{G})2^{|W|-1} \geq 2f_{\pi}(\tildeG).
\end{align*}
So $G'\in NC(n)$ and $P(G')=f_{\pi}(H)>f_{\pi}(\tilde{G})=P(G)$ imply $G\notin \calP(NC(n))$.  
\end{proof}

We now prove the main result of this subsection.

\begin{lemma}\label{extNCareinW}
For all $n\geq 1$, $\calP(NC(n))\cap W(n)\neq \emptyset$.  Consequently, $\calP(W(n))\subseteq \calP(NC(n))$.
\end{lemma}

\begin{proof}
If $n=1$, this is trivial.  If $n=2$ then the only element in $\calP(NC(n))$ is the $G$ which consists of a single edge with multiplicity $3$.  Clearly this $G$ also in $W(n)$.  Assume now $n\geq 3$ and let $G=([n],w)\in \calP(NC(n))$.  Suppose first that $G$ contains no edges of multiplicity $1$.  Then $\tilde{G}=([n],E)$ where $E=\{xy\in {[n]\choose 2}: w(xy)=3\}$.  By Lemma \ref{subgraph}, $\tilde{G}$ is a forest. It is a well known fact that because $\tilde{G}$ is a forest with $n$ vertices, $|E|\leq n-1$.  Therefore, we have that $P(G)=3^{|E|}2^{{n\choose 2}-|E|}\leq 3^{n-1}2^{{n\choose 2}-n+1}$. Let $G'=([n],w')$ be such that $w'(1i)=3$ for all $2\leq i\leq n$ and $w'(xy)=2$ for all other edges.  Then $G'\in NC(n)$ and $P(G')=3^{n-1}2^{{n\choose 2}-n+1}\geq P(G)$.  Since $G\in \calP(NC(n))$, this implies $G'\in \calP(NC(n))$ as well.  By definition, $G'\in W(n)$, so we are done.

Assume now $G$ contains some $xy$ with $w(xy)=1$.  Consider now the vertex-weighted graph $\tilde{G}=(\tilde{V},E,|\cdot|)$ associated to $G$ and $\sim_G$.  Suppose $(\tilde{V},E)$ is a star with center $V$.  If $|W|=1$ for all $W\in \tilde{V}\setminus \{V\}$, then $G\in W(n)$ and we are done.  If there is $W\in \tilde{V}\setminus \{V\}$ such that $|W|>1$, then Lemma \ref{wlemma} implies $G\notin \calP(NC(n))$, a contradiction.  

Suppose now $(\tilde{V},E)$ is not a star.  Then Lemma \ref{starlemma} implies there is a vertex-weighted graph $H=(\tilde{V}, E', |\cdot|)$ such that $(\tilde{V}, E')$ is a star and $f_{\pi}(H)\geq f_{\pi}(\tilde{G})$.  Since $(\tilde{V},E')$ is a star, it is a forest.  Since $(\tildeV, E, |\cdot|)$ is the vertex-weighted graph associated to $G$ and $\sim_{G}$, $\sum_{U\in \tildeV}|U|=n$.  Thus $H$ satisfies the hypotheses of Lemma \ref{vertexwhtlem}, so there is $G'\in NC(n)$ such that $\tilde{G'}\cong H$.  Thus $P(G')=f_{\pi}(H)\geq f_{\pi}(\tilde{G})$, where the equality holds by Definition \ref{vwdef}.  Suppose $f_{\pi}(H)>f_{\pi}(\tildeG)$.  Then
\begin{align*}
P(G')&=f_{\pi}(H)>f_{\pi}(\tilde{G})=P(G),
\end{align*}
contradicting that $G\in \calP(NC(n))$.  Thus we must have $f_{\pi}(H)= f_{\pi}(\tilde{G})$.  By Lemma \ref{starlemma}, this only happens if there is some $W\neq V\in \tilde{V}$ such that $|V|=|W|$, where $V$ is the center of the star $(\tilde{V},\tilde{E}')$.  Note that because $G$ contains some $xy$ with $w(xy)=1$, there is some vertex $U\in \tilde{V}$ such that $|U|>1$.  If $U\neq V$, then $U\in \tildeV\setminus \{V\}$ and $|U|>1$.  If $U=V$, then $W\in \tilde{V}\setminus \{V\}$ and $|W|=|V|=|U|>1$.  In either case Lemma \ref{wlemma} implies that $G'\notin \calP(NC(n))$.  Since $P(G)=f_{\pi}(\tilde{G})=f_{\pi}(H)=P(G')$, this implies $G\notin \calP(NC(n))$, a contradiction.  Thus we have shown that for all $n\geq 1$, $\calP(NC(n))\cap W(n)\neq \emptyset$.  Since $W(n)\subseteq NC(n)$, this implies $\calP(W(n))\subseteq \calP(NC(n))$.
\end{proof}

\subsection{Getting rid of cycles and proving Theorem \ref{EXTGINDAREINW}}\label{sectionextGinCareinNC}
In this subsection we prove Lemma \ref{extGinCareinNC}, which shows that for large $n$, all product-extremal elements of $C(n)$ are in $NC(n)$.  We will then prove Theorem \ref{EXTGINDAREINW} at the end of this subsection.  Our proof uses an argument that is essentially a progressive induction.

Note that if $G\in C(n)$, then $C_3(3,2)\nsubseteq G$ (since $C(n)\subseteq F(n,3,8)$) and $C_4(3,2)\nsubseteq G$ (since $S(C_4(3,2))=16$).  So to show some $G\in C(n)$ is in $NC(n)$, we only need to show $C_t(3,2)\nsubseteq G$ for $t\geq 5$.  Given $G=(V,w)$, $X\subseteq V$, and $z\in V\setminus X$, set $P_z^G(X)=\prod_{x\in X}w(xz)$.

\begin{lemma}\label{boundCtproduct}
Let $5\leq t\leq n$ and $G=([n],w)\in C(n)$.  Suppose $C_t(3,2)\subseteq G$, and for all $5\leq t'<t$, $C_{t'}(3,2)\nsubseteq G$.  If $X\in {[n]\choose t}$ is such that $G[X]\cong C_t(3,2)$, then for all $z\in [n]\setminus X$ either
\begin{enumerate}
\item $|\{x\in X : w(zx)=3\}|\leq 1$ and $P^G_z(X)\leq 3\cdot 2^{t-1}$ or
\item $|\{x\in X: w(zx)=3\}|\geq 2$ and $P^G_z(X)\leq 3^22^{t-3}<3\cdot2^{t-1}$.
\end{enumerate}
\end{lemma}

\begin{proof}
Let $X=\{x_1,\ldots, x_t\}$ where $w(x_i x_{i+1})=w(x_1x_t)=3$ for each $i\in [t-1]$ and $w(x_ix_j)=2$ for all other pairs $ij\in {[t]\choose 2}$.  Since $G\in C(n)$, $C_3(3,2), C_4(3,2)\nsubseteq G$.  Combining this with our assumptions, we have that for all $3\leq t'<t$, $C_{t'}(3,2)\nsubseteq G$.  We will use throughout that $\mu(G)\leq 3$ (since $G\in C(n)$).  Fix $z\in [n]\setminus X$ and let $Z=\{x\in X: w(zx)=3\}$.  If $|Z|\leq 1$, then clearly 1 holds.  So assume $|Z|\geq 2$ and $i_1<\ldots <i_{\ell}$ are such that $Z=\{x_{i_1},\ldots, x_{i_{\ell}}\}$.  Without loss of generality, assume $i_1=1$.  Set
$$
I=\{(x_{i_j},x_{i_{j+1}}):1\leq j\leq \ell-1\}\cup \{(x_{i_1},x_{i_{\ell}})\}.
$$
Given $(x,y)\in I$, let 
\[
d(x,y)=\begin{cases} i_{j+1}-i_j  & \text{ if }(x,y)=(x_{i_j},x_{i_{j+1}}) \text{ some }1\leq j\leq \ell-1\\
t-i_{\ell}+1 & \text{ if }(x,y)=(x_{i_1},x_{i_{\ell}}).
\end{cases}
\]
Note that because $C_3(3,2)\nsubseteq G$, $2\leq d(x,y)\leq t-2$ for all $(x,y)\in I$.  Suppose first that there is some $(u,v)\in I$ such that $d(u,v)=t-2$.  Then since $d(x,y)\geq 2$ for all $(x,y)\in I$ we must have that $|I|=1$ and either $(u,v)=(x_{i_1},x_{i_{\ell}})=(x_1,x_{t-1})$ or $(u,v)=(x_{i_1},x_{i_{\ell}})=(x_1,x_3)$.  Without loss of generality, assume $(u,v)=(x_1,x_3)$.  Then we must have that $w(zx_2)\leq 1$ since otherwise $G[\{z,x_1,x_2,x_3\}]\cong C_4(3,2)$, a contradiction.  This shows that $P^G_z(X)\leq 3^2\cdot 1\cdot 2^{t-3}<3\cdot 2^{t-1}$.

Suppose now that for all $(x,y)\in I$, $d(x,y)\leq t-3$.  Given $(x,y)\in I$, say an element $x_k$ is \emph{between} $x$ and $y$ if either $(x,y)=(x_{i_j},x_{i_{j+1}})$ and $i_j<k<i_{j+1}$ or $(x,y)=(x_{i_1},x_{i_{\ell}})$ and $i_{\ell}<k$.  Then for each $(x,y)\in I$, there must be a $x_k$ between $x$ and $y$ such that $w(zx_k)\leq1$, since otherwise  
$$
\{z,x,y\}\cup \{u: u \text{ is between }x\text{ and }y\}
$$ 
is a copy of $C_{d(x,y)+2}(3,2)$ in $G$, a contradiction since $d(x,y)+2<t$.   This implies there are at least $\ell$ elements $u$ in $X\setminus Z$ such that $w(zu)\leq 1$, so $P^G_z(X)\leq 3^{\ell} 2^{t-2\ell} \leq 3^22^{t-4}<3\cdot 2^{t-1}$.
\end{proof}

\noindent Given $n,t\in \mathbb{N}$ set
$$
f(n,t)=\min\Big\{ 2^{{\lceil  \beta t\rceil \choose 2}+\lceil  \beta t\rceil c}3^{\lceil \beta t\rceil \lfloor (1-\beta)t\rfloor+c\lfloor (1-\beta)t\rfloor+\lceil  \beta t\rceil (n-t-c)}: c\in \{\lfloor \beta (n-t)\rfloor, \lceil \beta (n-t)\rceil\}\Big\}.
$$

\begin{definition}\label{GXdef}
Suppose $t\leq n$ and $X\in {[n]\choose t}$.  Define $\mathcal{G}_X=([n],w)$ to be the following multigraph, where $Y=[n]\setminus X$. Choose any $A\in \calP(W(n-t))$ and $B\in W(t)$ so that $|R(B)|=\lceil \beta t\rceil$ and $|L(B)|=\lfloor (1-\beta )t\rfloor$.  Define $w$ on ${Y\choose 2}\cup {X\choose 2}$ to make $\mathcal{G}_X[Y]\cong A$ and $\mathcal{G}_X[X]\cong B$. Define $w$ on the remaining pairs of vertices in the obvious way so that $\mathcal{G}_X\in W(n)$.  
\end{definition}

\begin{lemma}\label{GXlem} Suppose $t\leq n$ and $X\in {[n]\choose t}$. Then for any $A'\in \calP(W(n-t))$,  $P(\mathcal{G}_X)\geq P(A')f(n,t)$.
\end{lemma}

\begin{proof} Set $Y=[n]\setminus X$ and let $\calG_X=([n],w)$.  Let $B\in W(t)$ and $A\in \calP(W(n-t))$ be as in the definition of $\calG_X$ so that $\calG_X[X]\cong B$ and $\calG_X[Y]\cong A$.  Let $L_A, R_A$ and $L_B, R_B$ be the partitions of $Y$ and $X$ respectively such that $w(xy)=1$ for all $xy\in {L_A\choose 2}\cup {L_B\choose 2}$.  By choice of $B$, $|L_B|=\lfloor (1-\beta) t \rfloor$ and $|R_B|=\lceil \beta t\rceil$.  Let $c=|R_A|$.  By definition, $|L_A|=n-t-c$, and since $A\in \calP(W(n-t))$, $c\in \{\lfloor \beta (n-t)\rfloor, \lceil \beta (n-t)\rceil\}$ (by the proof of Lemma \ref{extW}).  Combining these observations with the definition of $f(n,t)$ implies
\begin{align*}
2^{{|R_B|\choose 2}+|R_A||R_B|}3^{|L_B||R_B|+|R_B||L_A|+|L_B||R_A|}= 2^{{\lceil  \beta t\rceil \choose 2}+\lceil  \beta t\rceil c}3^{\lceil \beta t\rceil \lfloor (1-\beta)t\rfloor+c\lfloor (1-\beta)t\rfloor+\lceil  \beta t\rceil (n-t-c)}\geq f(n,t).
\end{align*}
Combining this with the definition of $\calG_X$, we have 
$$
P(\calG_{X})=P(A)2^{{|R_B|\choose 2}+|R_A||R_B|}3^{|L_B||R_B|+|R_B||L_A|+|L_B||R_A|}\geq P(A)f(n,t).
$$
Since $P(A)=P(A')$ for all $A'\in \calP(W(n-t))$, this finishes the proof.
\end{proof}

\begin{definition}
Given $n,t\in \mathbb{N}$, let $h(n,t)=3^n2^{{t\choose 2}+t(n-t)-n}$.
\end{definition}

\begin{lemma}\label{lemmallt}
Let $5\leq t\leq n$, $G\in C(n)$, and $\nu>0$.  Suppose $X\in {[n]\choose t}$, $G[X]\cong C_t(3,2)$, and there is some $A\in \calP(W(n-t))$ such that $P(G[[n]\setminus X])\leq \nu P(A)$.  Then $P(G)\leq \nu((h(n,t))/f(n,t))P(\mathcal{G}_X)$.
\end{lemma}
\begin{proof}
Let $Y=[n]\setminus X$.  Because $G[X]\cong C_t(3,2)$, $P(G)=P(G[Y])3^t2^{{t\choose 2}-t}\prod_{z\in Y}P^G_z(X)$.  By Lemma \ref{boundCtproduct}, for each $z\in Y$, $P^G_z(X)\leq 3\cdot 2^{t-1}$.  This implies 
\begin{align}\label{ineq0}
P(G)\leq P(G[Y])3^t2^{{t\choose 2}-t}\Big(3\cdot 2^{t-1}\Big)^{n-t}= P(G[Y])3^n2^{{t\choose 2}+t(n-t)-n}= P(G[Y])h(n,t).
\end{align}
By assumption, $P(G[Y])\leq \nu P(A)$, so (\ref{ineq0}) implies $P(G)\leq \nu P(A)h(n,t)$.  Combining this with Lemma \ref{GXlem} yields
$$
P(G)\leq \nu P(A)h(n,t)=\nu P(A)f(n,t)(h(n,t)/f(n,t))\leq \nu P(\mathcal{G}_X)(h(n,t)/f(n,t)).
$$
\end{proof}

\noindent The following will be proved in the Appendix.

\begin{lemma}\label{computationallem}
There are $\gamma>0$ and $5<K\leq M_1$ such that the following holds.
\begin{enumerate}
\item For all $K\leq t\leq n$, $h(n,t)< f(n,t)$.
\item For all $5\leq t\leq K$ and $n\geq M_1$, $h(n,t)<2^{-\gamma n}f(n,t)$.
\end{enumerate}
\end{lemma}

\begin{lemma}\label{bigt}
Let $K$ be from Lemma \ref{computationallem}.  Then for all $K\leq t\leq n$, the following holds.  If $G\in C(n)$, $C_t(3,2)\subseteq G$, and $C_{t'}(3,2)\nsubseteq G$ for all $t'<t$, then for all $G_1\in \calP(W(n))$, $P(G)<P(G_1)$.
\end{lemma}

\begin{proof}
Let $t\geq K$ and $n=t+i$.  We proceed by induction on $i$.  Suppose first $i=0$.  Fix $G\in C(n)$ such that $C_t(3,2)\subseteq G$ and $C_{t'}(3,2)\nsubseteq G$ for all $t'<t$.  Then $n=t$ implies $G\cong C_t(3,2)$ and so $P(G)=3^t2^{{t\choose 2}-t}=h(t,t)$.  Let $H\in W(n)$ have 
$|R(H)|\in \{\lceil \beta n\rceil, \lfloor \beta n \rfloor\}$ and $L(H)=V(H)-R(H)$.  Then by definition of $f(n,t)$,
\begin{align*}
P(H)&=2^{|R(H)|\choose 2}3^{|L(H)||R(H)|}= f(t,t)>h(t,t)=P(G),
\end{align*}
where the inequality is by part (1) of Lemma \ref{computationallem}.  Since $H\in W(n)$, this implies $P(G)<P(G_1)$ for all $G_1\in \calP(W(n))$. 

Suppose now that $i>0$.  Assume by induction that the conclusion of Lemma \ref{bigt} holds for all $K\leq t_0\leq n_0$ where $n_0=t_0+j$ and $0\leq j<i$.  Fix $G\in C(n)$ such that $C_t(3,2)\subseteq G$ and $C_{t'}(3,2)\nsubseteq G$ for all $t'<t$.  Let $X\in {[n]\choose t}$ be such that $G[X]\cong C_t(3,2)$.  
\begin{claim}\label{G/Xclaim}
For any $A\in \calP(W(n-t))$, $P(G[[n]\setminus X])\leq P(A)$.
\end{claim}
\par\noindent{\it Proof.}
Note that $C_{t'}(3,2)\nsubseteq G[[n]\setminus X]$ for all $3\leq t'<t$.  We have two cases.
\begin{enumerate}[(1)]
\item If $C_{t'}(3,2)\nsubseteq G[[n]\setminus X]$ for all $t'\geq t$, then $G[[n]\setminus X]$ is isomorphic to some $D\in NC(n-t)$.  By Lemma \ref{extNCareinW}, for any $A\in \calP(W(n-t))$, $P(G[[n]\setminus X])=P(D)\leq P(A)$.
\item If $C_{t'}(3,2)\subseteq G[[n]\setminus X]$ for some $t'\geq t$, then fix $t_0$ the smallest such $t'$, and set $n_0=n-t$.  Our assumptions imply $t_0\geq t\geq K$ and $t_0\leq |[n]\setminus X|=n-t=i$, so 
$$
n_0=n-t=t_0+(n-t-t_0)= t_0+i-t_0=t_0+j,
$$
where $0\leq i-t_0=j<i$.  Note $G[[n]\setminus X]$ is isomorphic to some $D\in C(n-t)=C(n_0)$.  Then we have that $K\leq t_0\leq n_0$, $n_0=t_0+j$, $0\leq j<i$, and $D\in C(n_0)$ satisfies $C_{t_0}(3,2)\subseteq D$ and $C_{t'}(3,2)\nsubseteq D$ for all $t'<t_0$. By our induction hypothesis, for any $A\in \calP(W(n_0))=\calP(W(n-t))$, $P(G[[n]\setminus X])=P(D)< P(A)$.
\end{enumerate}

\noindent Claim \ref{G/Xclaim} and Lemma \ref{lemmallt} with $\nu=1$ imply $P(G)\leq (h(n,t)/f(n,t))P(\mathcal{G}_X)$.  Since $K\leq t\leq n$, Lemma \ref{computationallem} part (1) implies $h(n,t)/f(n,t)<1$, so this shows $P(G)<P(\mathcal{G}_X)$.  Since $\mathcal{G}_X\in W(n)$, we have $P(G)< P(\mathcal{G}_X)\leq P(G_1)$ for all $G_1\in \calP(W(n))$.  
\end{proof}

\begin{lemma}\label{smallt}
Let $M_1$ and $K$ be as in Lemma \ref{computationallem}.  There is $M_2$ such that for all $5\leq t\leq K$ and $n\geq M_1+K$, the following holds. If $G\in C(n)$, $C_t(3,2)\subseteq G$, and $C_{t'}(3,2)\nsubseteq G$ for all $t'<t$, then for all $G_1\in \calP(W(n))$, $P(G)\leq 2^{M_2}(h(n,t)/f(n,t))P(G_1)$.
\end{lemma}
\begin{proof}
Set $M=M_1+K$.  Choose $M_2$ sufficiently large so that for all $5\leq t\leq K$ and $t\leq n\leq M$, $\expi(n,4,15)\leq 2^{M_2}(h(n,t)/f(n,t))$.  We show the conclusions of Lemma \ref{smallt} hold for all $n\geq M$ by induction.  Suppose first $n=M$. Fix $5\leq t\leq K$ and $G\in C(n)$ such that $C_t(3,2)\subseteq G$ and $C_{t'}(3,2)\nsubseteq G$ for all $t'<t$.  Then by our choice of $M_2$,
$$
P(G)\leq \expi(n,4,15)\leq 2^{M_2}(h(n,t)/f(n,t))\leq 2^{M_2}(h(n,t)/f(n,t))P(G_1),
$$
for all $G_1\in \calP(W(n))$.  Suppose now $n>M$.  Assume by induction the conclusions of Lemma \ref{smallt} hold for all $5\leq t_0\leq K$ and $M\leq n_0<n$.  Fix $5\leq t\leq K$ and $G\in C(n)$ such that $C_t(3,2)\subseteq G$ and $C_{t'}(3,2)\nsubseteq G$ for all $t'<t$.   Let $X\in {[n]\choose t}$ be such that $G[X]\cong C_t(3,2)$ and set $n_0=n-t$.
\begin{claim}\label{G/Xclaim2}
For any $A\in \calP(W(n_0))$, $P(G[[n]\setminus X])\leq 2^{M_2}P(A)$,
\end{claim}
\par\noindent{\it Proof.}
Fix $A\in \calP(W(n_0))$.  Note that $C_{t'}(3,2)\nsubseteq G[[n]\setminus X]$ for all $t'<t$ and $n_0\geq M_1\geq K$ (since $n-t\geq M-K=M_1$).  We will use the following observation. 
\begin{align}\label{goal}
\text{ For all $5\leq t_0\leq K$, $\frac{h(n_0,t_0)}{f(n_0,t_0)}\leq 2^{-\gamma n_0}<1$.}
\end{align}
This holds by Lemma \ref{computationallem} part (2) and the fact that $n_0\geq M_1$.  Suppose first $n_0<M$.  Then $G[[n]\setminus X]$ is isomorphic to some $D\in F(n_0,4,15)$ and $n_0\geq K$, so by our choice of $M_2$,
$$
P(G[[n]\setminus X])=P(D)\leq \expi(n_0,4,15)\leq 2^{M_2}\frac{h(n_0,K)}{f(n_0,K)}P(A)\leq 2^{M_2}P(A),
$$
where the last inequality is by (\ref{goal}).  Assume now $n_0\geq M$.  We have two cases.
\begin{enumerate}[(1)]
\item If $C_{t'}(3,2)\nsubseteq G[[n]\setminus X]$ for all $t'\geq t$, then $G[[n]\setminus X]$ is isomorphic to some $D\in NC(n_0)$.  By Lemma \ref{extNCareinW}, $P(G[[n]\setminus X])=P(D)\leq P(A)\leq 2^{M_2}P(A)$.
\item If $C_{t'}(3,2)\subseteq G[[n]\setminus X]$ for some $t'\geq t$, choose $t_0$ the smallest such $t'$, and let $D\in C(n_0)$ be such that $G[[n]\setminus X]\cong D$.  Suppose first $t_0\leq K$.  Then we have $5\leq t\leq t_0\leq K$, $M\leq n_0<n$, $D\in C(n_0)$, $C_{t_0}(3,2)\subseteq D$, and $C_{t'}(3,2)\nsubseteq D$ for all $t'<t_0$.  Therefore our induction hypothesis implies the conclusions of Lemma \ref{smallt} hold for $D$, $n_0$, $t_0$. In other words, since $A\in \calP(W(n_0))$,
$$
P(G[[n]\setminus X])=P(D)\leq 2^{M_2}\frac{h(n_0,t_0)}{f(n_0,t_0)}P(A)\leq 2^{M_2}P(A),
$$
where the last inequality is by (\ref{goal}).   Suppose finally that $t_0>K$.  Then $K\leq t_0\leq n_0$, $D\in C(n_0)$, $C_{t_0}(3,2)\subseteq D$, and $C_{t'}(3,2)\nsubseteq D$ for all $t'<t_0$.  Thus we have by Lemma \ref{bigt} that $P(G[[n]\setminus X])=P(D)<P(A)\leq 2^{M_2}P(A)$.
\end{enumerate}

\noindent Claim \ref{G/Xclaim2} and Lemma \ref{lemmallt} with $\nu=2^{M_2}$ imply $P(G)\leq 2^{M_2}(h(n,t)/f(n,t))P(\mathcal{G}_X)$.  Since $\mathcal{G}_X$ is in $W(n)$, we have that $P(G)\leq 2^{M_2}(h(n,t)/f(n,t))P(\mathcal{G}_X)\leq 2^{M_2}(h(n,t)/f(n,t))P(G_1)$, for all $G_1\in \calP(W(n))$.
\end{proof}

We can now prove that for large $n$, all product-extremal elements of $C(n)$ are in $NC(n)$.
\begin{lemma}\label{extGinCareinNC}
For all sufficiently large $n$, $\calP(C(n))\subseteq NC(n)$.  Consequently, $\calP(C(n))=\calP(NC(n))$.
\end{lemma}
\begin{proof}
Let $\gamma$, $K$, and $M_1$ be as in Lemma \ref{computationallem} and let  $M_2$ be as in Lemma \ref{smallt}.  Choose $M\geq M_1+K$ sufficiently large so that $2^{M_2-\gamma n}<1$ for all $n\geq M$.  Suppose $n>M$ and $G\notin NC(n)$.  We show $G\notin \calP(C(n))$.  Clearly if $G\notin C(n)$ we are done, so assume $G\in C(n)$.  Since $W(n)\subseteq C(n)$, it suffices to show there is $G_1\in W(n)$ such that $P(G_1)>P(G)$.  Since $G\notin NC(n)$, there is $5\leq t\leq n$ such that $C_t(3,2)\subseteq G$ and for all $t'<t$, $C_{t'}(3,2)\nsubseteq G$.  If $t\geq K$, then Lemma \ref{bigt} implies that for any $G_1\in \calP(W(n))$, $P(G)<P(G_1)$.  If $5\leq t<K$, then Lemma \ref{smallt} implies that for any $G_1\in \calP(W(n))$, 
$$
P(G)\leq 2^{M_2}(h(n,t)/f(n,t))P(G_1)\leq 2^{M_2-\gamma n}P(G_1),
$$
where the second inequality is because of Lemma \ref{computationallem} part (2). By our choice of $M$, this implies that for all $G_1\in W(n)$, $P(G)<P(G_1)$.  This shows $\calP(C(n))\subseteq NC(n)$.  Since $NC(n)\subseteq C(n)$, this implies $\calP(C(n))=\calP(NC(n))$.
\end{proof}

We can prove the main result of this section, Theorem \ref{EXTGINDAREINW}.

\vspace{3mm}

\noindent {\bf Proof of Theorem \ref{EXTGINDAREINW}.} Assume $n$ sufficiently large.  By Lemma \ref{EXTGINBCAPAAREINC}, we can choose some $G$ in $ \calP(D(n))\cap C(n)=\calP(C(n))$. By Lemma \ref{extGinCareinNC}, $\calP(C(n))=\calP(NC(n))$, so $G\in \calP(NC(n))$.  By Lemma \ref{extNCareinW}, there is some $G'\in \calP(NC(n))\cap W(n)=\calP(W(n))$.  Since $G$ and $G'$ are both in $\calP(NC(n))$, $P(G)=P(G')$. Since $G\in \calP(D(n))$ and $W(n)\subseteq D(n)$, this implies that $G'\in \calP(D(n))$.  Thus we have shown $G'\in \calP(D(n))\cap \calP(W(n))$.
\qed

\section{Proof of Theorem \ref{EXTGAREIND}}\label{sectionEXTGAREIND}
In this section we prove Theorem \ref{EXTGAREIND}.  We will need the following computational lemma, which is proved in the appendix.  Given $n,t$, let $k(n,t)=15^t2^{{t\choose 2}+t(n-t)-t}$.

\begin{lemma}\label{complem2}
There is $M$ such that for all $n\geq M$ and $2\leq t\leq n$, $k(n,t)<f(n,t)$.
\end{lemma}

The following can be checked easily by hand and is left to the reader.

\begin{lemma}\label{arithgeom}
Suppose $a$, $b$, and $c$ are non-negative integers.  If $a+b\leq 4$, then $a\cdot b\leq 2^2$.  If $a+b+c\leq 6$, then $a\cdot b \cdot c\leq 2^3$.
\end{lemma}

\noindent {\bf Proof of Theorem \ref{EXTGAREIND}.} Let $n$ be sufficiently large.  It suffices to show $\calP(n,4,15)\subseteq D(n)$.  Suppose towards a contradiction there is $G=([n],w)\in \calP(F(n,4,15))\setminus D(n)$.  Given $X\subseteq [n]$ and $z\in [n]\setminus X$, let $S(X)=S(G[X])$ and $S_z(X)=\sum_{x\in X}w(xz)$.  If $G\notin F(n,3,8)$, let $D_1,\ldots, D_k$ be a maximal collection of pairwise disjoint elements of ${[n]\choose 3}$ such that $S(D_i)\geq 9$ for each $i$, and set $D=\bigcup_{i=1}^k D_i$.  If $G\in F(n,3,8)$, set $D=\emptyset$.  If $\mu(G[[n]\setminus D])>3$, choose $e_1,\ldots, e_m$ a maximal collection of  pairwise disjoint elements of ${[n]\setminus D\choose 2}$ such that $S(e_i)\geq 4$ for each $i$ and set $C=\bigcup_{i=1}^m e_i$.  If $\mu(G[[n]\setminus D])\leq 3$, set $C=\emptyset$.  Let $X=D\cup C$ and $\ell=|X|=3k+2m$.  Note that by assumption $X$ is nonempty, so we must have $\ell \geq 2$.  We now make a few observations.  If $D\neq \emptyset$, then for each $D_i$ and $z\in [n]\setminus D_i$, 
$$
S_z(D_i)\leq S(D_i\cup \{z\})-S(D_i)\leq 15-9=6=2\cdot 3,
$$
which implies by Lemma \ref{arithgeom} that  $P^G_z(D_i)\leq 2^3$.  By maximality of the collection $D_1,\ldots, D_k$, $G[[n]\setminus D]$ is a $(3,8)$-graph.  Thus if $C\neq \emptyset$, then for each $i$ and $z\in [n]\setminus (D\cup e_i)$, 
$$
S_z(e_i)\leq S(e_i\cup \{z\})-4 \leq 8-4=4=2\cdot 2,
$$
which implies by Lemma \ref{arithgeom} that $P^G_z(e_i)\leq 2^2$.  Since $\mu(G)\leq 15$, for each $D_i$ and $e_j$, $P(D_i)\leq 15^3$ and $P(e_j)\leq 15$.  Let $Y=[n]\setminus X$ and write $P(Y)$ for $P(G[Y])$.  Our observations imply that $P(G)$ is at most 
\begin{align}\label{ineq5}
P(Y)\Big(\prod_{i=1}^kP(D_i)\Big)\Big(\prod_{i=1}^mP(e_i)\Big)2^{{\ell\choose 2}+\ell(n-\ell)-\ell+m}\leq P(Y)15^{\ell-m}2^{{\ell\choose 2}+\ell(n-\ell)-\ell+m} \leq P(Y)k(n,\ell).
\end{align}
Note that $G[Y]$ is isomorphic to an element of $D(n-\ell)$.   Let $n_0$ be such that Lemma
\ref{extGinCareinNC} holds for all $n>n_0$.  We partition the argument into two cases.
\medskip

\noindent {\bf Case 1.} $n-\ell\le n_0$. In this case we can use the crude bounds 
$$P(G) < 2^{ {\ell \choose 2}} 15^{\ell-m+{n_0 \choose 2}}2^{\ell n_0}
< 2^{{\ell \choose 2}+4\ell+2n_0^2+\ell n_0} < \expi(n,4,15)$$
where the last inequality holds since we may assume that $n$ is much larger than $n_0$ and  $\ell>n-n_0$.  This contradicts the fact that $G \in \calP(n,4,15)$.

\medskip

\noindent {\bf Case 2.} $n-\ell> n_0$. In this case may  apply 
 Lemma \ref{extGinCareinNC} to $G[Y]$ as $|Y|=n-\ell>n_0$. 
 Fix $A\in \calP(W(n-\ell))$. By Lemma \ref{EXTGINBCAPAAREINC}, Lemma \ref{extGinCareinNC}, and Lemma \ref{extNCareinW}, $\calP(W(n-\ell))\subseteq \calP(D(n-\ell))$, which implies that $P(Y)\leq P(A)$.  Combining this with Lemma \ref{GXlem} yields $P(\mathcal{G}_X)\geq P(A)f(n,\ell)\geq P(Y)f(n,\ell)$.  This, along with the bound on $P(G)$ in (\ref{ineq5}), implies
\begin{align*}
\frac{P(G)}{P(\mathcal{G}_X)}\leq \frac{P(Y)k(n,\ell)}{P(Y)f(n,\ell)}=\frac{k(n,\ell)}{f(n,\ell)}<1,
\end{align*}
where the last inequality is by choice of $M$ and Lemma \ref{complem2}.  So $P(G)<P(\mathcal{G}_X)$, a contradiction.  
\qed

\vspace{2mm}

\section{Concluding Remarks} \label{con}

The arguments used to prove Theorem \ref{caseiv} can be adapted to prove a version for sums.  If $G=(V,w)$, let $S(G)=\sum_{xy\in {V\choose 2}}w(xy)$.  Given integers $s\geq 2$ and $q\geq 0$, set 
$$
\exs(n,s,q)=\max\{S(G): G\in F(n,s,q)\}.
$$
An $(n,s,q)$-graph $G$ is \emph{sum-extremal} if $S(G)=\exs(n,s,q)$. Let $\calS(n,s,q)$ denote the set of sum-extremal $(n,s,q)$-graphs with vertex set $[n]$, and let $\calS(W(n))$ denote the set of $G\in W(n)$ such that $S(G)\geq S(G')$ for all $G'\in W(n)$.  Straightforward calculus shows that for $G\in W(n)$, the sum $S(G)$ is maximized when $|L(G)|\approx (2/3)n$.   Then our proofs can be redone for sums to obtain the following theorem.

\begin{theorem}
For all sufficiently large $n$, $\calS(W(n))\subseteq \calS(n,4,15)$.  Consequently 
$$
\exs(n,4,15)=\max\Big\{ 2{\lfloor \frac{2n}{3}\rfloor \choose 2}+3\Big(\Big\lfloor \frac{2n}{3}\Big\rfloor\Big)\Big(\Big\lceil \frac{n}{3}\Big\rceil\Big),2{\lceil \frac{2n}{3}\rceil \choose 2}+3\Big(\Big\lceil \frac{2n}{3}\Big\rceil\Big)\Big(\Big\lfloor \frac{n}{3}\Big\rfloor\Big)\Big\}=\frac{8}{3}{n\choose 2}+O(n).
$$
\end{theorem}
We would like to point out that the asymptotic value for $\exs(n,4,15)$ was already known as a consequence of \cite{furedikundgen}.  Our contribution is in showing $\calS(W(n))\subseteq \calS(n,4,15)$.  The following result  shows that product-extremal $(n,4,15)$-graphs are far from sum-extremal ones.

\begin{corollary}
There is $\delta>0$ such that for all sufficiently large $n$, the following holds.  Suppose $G\in \calP(n,4,15)$ and $G'\in \calS(n,4,15)$.  Then $G$ and $G'$ are $\delta$-far from one another.
\end{corollary}

\begin{proof}
Assume $n$ is sufficiently large and $\delta$ is sufficiently small. Suppose towards a contradiction that $G\in \calP(n,4,15)$ and $G'\in \calS(n,4,15)$ are $\delta$-close.  Since $\mu(G),\mu(G')\leq 15$, this implies
\begin{align}\label{slab}
S(G)\geq S(G')-15|\Delta(G,G')| \geq S(G')-15\delta n^2.
\end{align}
Using the asymptotic value of $\exs(n,4,15)$, this implies $S(G)\geq \frac{8}{3}{n\choose 2}-15\delta n^2$.  On the other hand, fix $H\in \calP(W(n))$ and let $L=L(H)$ and $R=R(H)$. Theorem \ref{caseiv} implies $P(G)=P(H)$.  Note Theorem \ref{EXTGAREIND} implies that $\mu(G)\leq 3$.  Thus $P(G)=P(H)=2^{{|R|\choose 2}}3^{|L||R|}=2^{|E_2(G)|}3^{|E_3(G)|}$  (where $E_i(G)$ is the set of edges of multiplicity $i$ in $G$).  Since $2$ and $3$ are relatively prime, this implies $|E_2(G)|={|R|\choose 2}$, $|E_3(G)|=|L||R|$, and $|E_1(G)|={|L|\choose 2}$.  So 
\begin{align*}
S(G)&={|L|\choose 2}+2{|R|\choose 2}+3|L||R|={n\choose 2}+{|R|\choose 2}+2|L||R|.
\end{align*}
Because $H\in \calP(W(n))$, $|R(H)|\leq \beta n+1$ and $|L(H)|\leq (1-\beta )n +1$.  Therefore 
\begin{align*}
S(G)&\leq {n\choose 2}+{\beta n+1\choose 2}+2(\beta n+1)((1-\beta )n+1)=n^2\Big(\frac{1}{2}+2\beta -\frac{3}{2}\beta^2\Big)-n\Big(\frac{4+\beta }{2}\Big)+2.
\end{align*}
But a straightforward computation shows $\frac{1}{2}+2\beta-3\beta^2/2<8/6$, so since $n$ is large and $\delta$ is small,
$$
S(G)<n^2\Big(\frac{1}{2}+2\beta -\frac{3}{2}\beta^2\Big)<\frac{8}{3}{n\choose 2}-15\delta n^2,
$$
contradicting (\ref{slab}).
\end{proof}

Given $a\geq 2$, let $W_a(n)$ be the set of multigraphs $([n],w)$ such that there is a partition $L,R$ of $[n]$ with $w(xy)=a-1$ for all $xy\in {L\choose 2}$, $w(xy)=a$ for all $xy\in {R\choose 2}$, and $w(xy)=a+1$ for all $x\in L$, $y\in R$.  Basic calculus shows that for $G\in W_a(n)$, $P(G)$ is maximized when $|R|\approx \beta_a n$ where $\beta_a=\frac{\log(a+2)-\log(a-1)}{2\log(a+2)-\log a -\log (a-1)}$.  Note that the $W(n)=W_2(n)$.  Based on our results for $(4,15)$, we make the following conjecture.
\begin{conjecture} \label{a}
For all $a\geq 2$, $\calP(W_a(n))\subseteq \calP(n,4,6a+3)$.  Consequently, 
$$
\expi(n,4,6a+3)=2^{\gamma_a n^2+O(n)},
$$
 where $\gamma_a=\frac{(1-\beta_a)^2}{2}\log_2(a-1)+ \frac{\beta_a^2}{2}\log_2 a+\beta_a(1-\beta_a)\log_2(a+2)$.
\end{conjecture}
When $a=2$, this is Theorem \ref{caseiv}.  However, at least some of the arguments used in this paper will not transfer immediately to cases with $a>2$.  For instance, the proof of Lemma \ref{triangles1} uses the fact that $a=2$ in a nontrivial way (in particular it is key there that the smallest multiplicity appearing in $W(n)$ is $1$).  Further, when $a>2$, one must contend with ``small'' edge multiplicities, that is, those in $\{i: 1\leq i< a-2\}$.  This is not an issue for $(4,15)$ since this set is empty.

\section{Acknowledgments}
We wish to thank Ping Hu for doing some Flag Algebra computations for us. The first author thanks Sasha Razborov for some helpful short remarks about presentation and the efficacy of the methods in~\cite{Razborov}.

\section{Appendix}

For ease of notation, we will write $x=\beta $ for the rest of this section.  For any $r\in\mathbb{R}$, ${r\choose 2}=\frac{r^2-r}{2}$.  Recall that given $n,t\in \mathbb{N}$
$$
f(n,t)=\min\Big\{ 2^{{\lceil  \beta t\rceil \choose 2}+\lceil  \beta t\rceil c}3^{\lceil \beta t\rceil \lfloor (1-\beta)t\rfloor+c\lfloor (1-\beta)t\rfloor+\lceil  \beta t\rceil (n-t-c)}: c\in \{\lfloor \beta (n-t)\rfloor, \lceil \beta (n-t)\rceil\}\Big\}.
$$

\noindent Given $2\leq t\leq n$, let 
$$
f_{*}(n,t)=2^{{xt\choose 2}+x^2t(n-t)}3^{2xt(1-x)(n-t)+x(1-x)t^2}.
$$  

\begin{proposition}\label{prop1}
For all $2\leq t\leq n$, $f(n,t)\geq f_*(n,t)2^{-xt-3/2}3^{-t-1}$.
\end{proposition}
\begin{proof}
By definition of $x$, $x(2\log 3-\log 2 )=\log3$.  Dividing both sides of this by $\log2$ and rearranging yields 
\begin{align}\label{fact1}
-x-\log_23+2x\log_23=0.
\end{align}
Fix $2\leq t\leq n$ and let $a=\lceil xt\rceil -xt$.  Define $\eta(u,v,z,w)=2^{{u\choose 2}+uz}3^{uw+vz+uv}$ and observe that
\begin{align}
f(n,t)=&\min \{ \eta(\lceil x t \rceil, \lfloor (1-x)t\rfloor, y, n-t-y): y\in \lceil x (n-t)\rceil, \lfloor x (n-t)\rfloor\}\}\nonumber\\
=&\min \{ \eta(xt+a,(1-x)t-a, y, n-t-y): y\in \lceil x (n-t)\rceil, \lfloor x (n-t)\rfloor\}\}.\label{app}
\end{align}
Note that for all $y\in \{ \lceil x (n-t)\rceil, \lfloor x (n-t)\rfloor\}$, $y\geq x(n-t)-1$ and $n-t-y\geq (1-x)(n-t)-1$.  Combining this with (\ref{app}) and the definition of $\eta(u,v,z,w)$, we have
\begin{align}\label{line}
f(n,t)\geq \eta(xt+a, (1-x)t-a, x(n-t)-1, (1-x)(n-t)-1).
\end{align}
We leave it to the reader to verify that the righthand side of (\ref{line}) is equal to $f_{*}(n,t)2^{g_1(n,t)}3^{g_2(n,t)}$, where $g_1(n,t)=\frac{a^2}{2}-\frac{3a}{2}-xt+axn$ and $g_2(n,t)=-2axn+an-t-a^2$.
Observe 
\begin{align*}
g_1(n,t)+g_2(n,t)\log_23&=an\Big(x+\log_23-2x\log_23\Big)+\frac{a^2}{2}-\frac{3a}{2}-xt-(t+a^2)\log_23\\
&=\frac{a^2}{2}-\frac{3a}{2}-xt-(t+a^2)\log_23,
\end{align*}
where the second equality is by (\ref{fact1}).  Since $0\leq a\leq 1$, $\frac{a^2}{2}-\frac{3a}{2}=\frac{a}{2}(a-3)\geq \frac{a}{2}(-3)\geq -3/2$ and $-a^2\geq -1$.  So 
$$
g_1(n,t)+g_2(n,t)\log_23\geq -\frac{3}{2}-xt-(t+1)\log_23.
$$
Thus $f(n,t)\geq f_{*}(n,t)2^{g_1(n,t)}3^{g_2(n,t)}\geq f_{*}(n,t)2^{-\frac{3}{2}-xt}3^{-t-1}$, as desired.
\end{proof}

Recall that given $n,t\in \mathbb{N}$, let $h(n,t)=3^n2^{{t\choose 2}+t(n-t)-n}$.

\begin{proposition}\label{prop2}
Let $2\leq t\leq n$.  Then $h(n,t)/f(n,t)\leq 2^{C_1(n,t)}3^{C_2(n,t)}$, where 
$$
C_1(n,t)=\frac{t^2}{2}(x^2-1)+\frac{t}{2}(3x-1)+tn(1-x^2)-n+\frac{3}{2}\quad \text{and}\quad C_2(n,t)=n-2x(1-x)tn+x(1-x)t^2+t+1.
$$
\end{proposition}
\begin{proof}
Fix $2\leq t\leq n$.  Proposition \ref{prop1} and the definition of $h(n,t)$ implies 
\begin{align}\label{line2}
\frac{h(n,t)}{f(n,t)}\leq \frac{3^n2^{{t\choose 2}+t(n-t)-n}}{f_{*}(n,t)2^{-3/2-xt}3^{-t-1}}.
\end{align}
Plugging in $f_*(n,t)$ to the right hand side of (\ref{line2}) yields that $h(n,t)/f(n,t)\leq 2^{C_1(n,t)}3^{C_2(n,t)}$ where 
\begin{align*}
C_1(n,t)&={t\choose 2}+t(n-t)-n-\Big({xt\choose 2}+x^2t(n-t)-3/2-xt\Big)\text{ and }\\
C_2(n,t)&=n-\Big(x(1-x)t^2+2x(1-x)t(n-t)-t-1\Big).
\end{align*}
Simplifying these expressions finishes the proof.
\end{proof}

\noindent We now prove the following three inequalities.
\begin{enumerate}[(I)]
\item $2^{1-x^2}<3^{1.5x(1-x)}$. \label{line3}
\item $3^{(2/3)x(1-x)}<2^{(1-x^2)/2}$.\label{line4}
\item $5(1-x^2-2x(1-x)\log_23)+\log_23-1<0$.\label{line5}
\end{enumerate}
We will use the following bounds for $\log2$ and $\log3$ which come from the On-Line Encyclopedia of Integer Sequences, published electronically at \text{http://oeis.org} (Sequences A002162 and A002391 respectively).
\begin{align}
.693<\log 2<.694 \qquad \text{ and }\qquad 1.098<\log 3< 1.099 \label{log3}.
\end{align}

\noindent For (I), note that $2^{1-x^2}=2^{(1-x)(1+x)}<3^{1.5x(1-x)} \Leftrightarrow 2^{1+x}<3^{1.5x}\Leftrightarrow (1+x)\log 2<1.5x\log 3$.  Solving for $x$ yields that this is equivalent to
\begin{align}\label{(I)}
\frac{\log 2}{1.5\log 3-\log 2}=\frac{2\log 2}{3\log 3-2\log 2}<x=\frac{\log3}{2\log 3-\log2}.
\end{align}
Clearing out the denominators, (\ref{(I)}) holds if and only if 
\begin{align}\label{ineq7}
4\log 3\log2-2(\log 2)^2<3(\log 3)^2-2\log 2\log 3 \Leftrightarrow 6\log2\log3-3(\log3)^2-2(\log2)^2<0.
\end{align}
By (\ref{log3}), $6\log2\log3-3(\log3)^2-2(\log2)^2<6(.694)(1.099)-3(1.098)^2-2(.693)^2<0$.  Thus the righthand inequality in (\ref{ineq7}) holds, which finishes the proof of (\ref{line3}).  For (\ref{line4}), note that
$$
3^{(2/3)x(1-x)}<2^{(1-x^2)/2}=2^{(1-x)(1+x)/2} \Leftrightarrow 3^{2x/3}<2^{(1+x)/2}\Leftrightarrow \frac{2x}{3}\log 3< \frac{(1+x)\log 2}{2}.
$$
Rearranging and plugging in for $x$, this becomes 
$$
\frac{\log 3}{2\log 3-\log 2}=x<\frac{3\log 2}{4\log 3-3\log 2}.
$$
By clearing denominators, we have that this inequality holds if and only if 
\begin{align}\label{lg}
4(\log 3)^2-3\log 3\log 2<6\log 3\log 2-3(\log 2)^2\Leftrightarrow 4(\log3)^2-9\log2\log3+3(\log2)^2<0.
\end{align}
By (\ref{log3}), $4(\log3)^2-9\log2\log3+3(\log2)^2<4(1.099)^2-9(.693)(1.098)+3(.694)^2<0$.  Thus the righthand inequality in (\ref{lg}) holds, which finishes the proof of (\ref{line4}).   We now prove (\ref{line5}).  By rearranging the left hand side, (\ref{line5}) is equivalent to
$$
5x^2(2\log_23-1)-10x\log_23+\log_23+4<0.
$$
Multiplying by $\log 2$, this becomes $5x^2(2\log3-\log2)-10x\log3 +\log3+4\log2<0$.  Plugging in for $x$ and simplifying, this is equivalent to
\begin{align}\label{lg1}
\frac{-5(\log3)^2}{2\log3-\log2}+\log3+4\log2<0\Leftrightarrow -3(\log3)^2+7\log2\log3-4(\log2)^2<0,
\end{align}
where the ``$\Leftrightarrow$'' is from clearing the denominators of, then rearranging the lefthand inequality.  By (\ref{log3}), $-3(\log3)^2+7\log2\log3-4(\log2)^2<-3(1.098)^2+7(.694)(1.099)-4(.693)^2<0$, thus the righthand inequality in (\ref{lg1}) holds, which finishes the proof of (\ref{line5}).

\vspace{5mm}

\noindent{\bf Proof of Lemma \ref{computationallem}.}
Given $n,t\in \mathbb{N}$, let $p(n,t)=(-\frac{x}{6}(1-x)t+2)n+2$.  Choose $K$ sufficiently large so that $n\geq t\geq K$ implies $p(n,t)\leq p(n,K)<0$.  We now prove part 1 for this $K$.  Fix $K\leq t\leq n$.  By Proposition \ref{prop2}, $h(n,t)/f(n,t)\leq 2^{C_1(n,t)}3^{C_2(n,t)}$.  Note that 
$$
C_1(n,t)=(1-x^2)tn+D_1(n,t)\qquad \hbox{ and } \qquad C_2(n,t)=-1.5x(1-x)tn+D_2(n,t)
$$
where $D_1(n,t)=\frac{t^2}{2}(x^2-1) +\frac{t}{2}(3x-1)-n+3/2$ and $D_2(n,t)=-.5x(1-x)tn+ x(1-x)t^2+n+t+1$. 
Therefore
\begin{align*}
2^{C_1(n,t)}3^{C_2(n,t)} =\Big(\frac{2^{1-x^2}}{3^{1.5x(1-x)}}\Big)^{tn}2^{D_1(n,t)}3^{D_2(n,t)}\leq 2^{D_1(n,t)}3^{D_2(n,t)},
\end{align*}
where the inequality is because by (\ref{line3}), $\frac{2^{1-x^2}}{3^{1.5x(1-x)}}\leq 1$.  Now note that 
$$
D_1(n,t)=\frac{t^2}{2}(x^2-1) +E_1(n,t) \quad \hbox{ and }\quad D_2(n,t)= -(x/3)(1-x)tn+x(1-x)t^2+E_2(n,t),
$$
where $E_1(n,t)=\frac{t}{2}(3x-1)-n+3/2$ and $E_2(n,t)=-(x/6)(1-x)tn +n+t+1$. Since $n\geq t$, we have
$$
-(x/3)(1-x)tn+x(1-x)t^2\leq -(x/3)(1-x)t^2+x(1-x)t^2= (2x/3)(1-x)t^2,
$$
so $D_2(n,t)\leq (2x/3)(1-x)t^2+E_2(n,t)$.
Thus
\begin{align*}
2^{D_1(n,t)}3^{D_2(n,t)} \leq \Big(\frac{3^{(2/3)x(1-x)}}{2^{(1-x^2)/2}}\Big)^{t^2}2^{E_1(n,t)}3^{E_2(n,t)}\leq 2^{E_1(n,t)}3^{E_2(n,t)},
\end{align*}
where the last inequality is because by (\ref{line4}), $\frac{3^{(2/3)x(1-x)}}{2^{(1-x^2)/2}}\leq 1$.  Note that since $3x-1<2$, $n\geq t$ and $3/2\leq \log_23$,
$$
E_1(n,t)= \frac{t}{2}(3x-1)-n+3/2\leq t-t+\log_23=\log_23.
$$
Since $5\leq t\leq n$, $E_2(n,t)=-(x/6)(1-x)tn +n+t+1 \leq -(x/6)(1-x)tn +2n+1$.  Therefore,
\begin{align*}
2^{E_1(n,t)}3^{E_2(n,t)}\leq 2^{\log_23}3^{-(x/6)(1-x)tn +2n+1}=3^{-(x/6)(1-x)tn +2n+2}=3^{p(n,t)}<1.
\end{align*}
where the inequality is by assumption on $K\leq t\leq n$.  This finishes the proof of part 1. We now prove part 2.  By definition of $C_1(n,t)$ and $C_2(n,t)$, there are polynomials $q_1(t)$ and $q_2(t)$ such that
$$
C_1(n,t)=tn(1-x^2) -n+q_1(t)\qquad \hbox{ and }\qquad C_2(n,t)=n-2x(1-x)tn+q_2(t).
$$
Set $q(t)=q_1(t)+q_2(t)\log_23$ and choose $T$ so that for all $5\leq t\leq K$, $|q(t)|\leq T$.  Set 
$$
\gamma=-\frac{1}{2}\Big(5(1-x^2-2x(1-x)\log_23)+\log_23-1\Big).
$$
Observe that (\ref{line5}) implies $\gamma>0$.  Choose $M_1\geq K$ so that for all $5\leq t\leq K$, $-2\gamma n+T\leq -\gamma n$.  We show $\frac{h(n,t)}{f(n,t)}<2^{-\gamma n}$ for all $5\leq t\leq K$ and $n\geq M_1$.  Fix $5\leq t\leq K$ and $n\geq M_1$.  By Proposition \ref{prop2} and the definitions of $q(t)$ and $T$,
$$
\frac{h(n,t)}{f(n,t)}\leq 2^{C_1(n,t)}3^{C_2(n,t)}=2^{n(t(1-x^2-2x(1-x)\log_23)+\log_23-1)+q(t)}\leq 2^{n(t(1-x^2-2x(1-x)\log_23)+\log_23-1)+T}.
$$
By (\ref{line5}), $(1-x^2-2x(1-x)\log_23)<1-\log_23<0$ so since $t\geq 5$, 
$$
t(1-x^2-2x(1-x)\log_23)+\log_23-1\leq 5(1-x^2-2x(1-x)\log_23)+\log_23-1= -2\gamma.
$$
Combining all this yields $\frac{h(n,t)}{f(n,t)}\leq 2^{-2\gamma n+T}<2^{-\gamma n}$, where the last inequality is by choice of $M_1$.\qed
 
\vspace{4mm}

\noindent{\bf Proof of Lemma \ref{complem2}.}
Recall we want to show there is $M$ such that for all $n\geq M$ and $2\leq t\leq n$, $k(n,t)<f(n,t)$, where $k(n,t)=15^t2^{{t\choose 2}+t(n-t)-t}$.  Let $K$ be from Lemma \ref{computationallem} and recall the proof of Lemma \ref{computationallem} showed that for all $K\leq t\leq n$, $h(n,t)/f(n,t)\leq 3^{p(n,t)}$, where 
$$
p(n,t)=-(x/6)(1-x)tn+2n+2.
$$
Choose $K'\geq K$ such that $K'\leq t\leq n$ implies $p(n,t)<-100n+2<-98n$.  Suppose now that $K'\leq t\leq n$.  Then by definition of $k(n,t)$ and since $h(n,t)/f(n,t)\leq 3^{p(n,t)}<3^{-98n}$, 
$$
\frac{k(n,t)}{f(n,t)}=\frac{(15/2)^t(2/3)^nk(n,t)}{f(n,t)}\leq (15/2)^t(2/3)^n3^{p(n,t)}\leq 3^{p(n,t)+4n}<3^{-94n}<1.
$$
Thus the Lemma holds for all $K'\leq t\leq n$.  Suppose now that $2\leq t\leq K'$ and $n\geq t$.  By Proposition \ref{prop1} and definition of $k(n,t)$,
$$
\frac{k(n,t)}{f(n,t)}\leq \frac{15^t2^{{t\choose 2}+t(n-t)-t}}{f_*(n,t)2^{-xt-3/2}3^{-t-1}}=2^{G_1(n,t)}3^{G_2(n,t)},
$$
where $G_1(n,t)$ and $G_2(n,t)$ are the appropriate polynomials in $n$ and $t$.  Using the definition of $f_*(n,t)$, we see that for some polynomials $r_1(t)$ and $r_2(t)$ in $t$,
$$
G_1(n,t)=tn-x^2 tn +r_1(t)\qquad \hbox{ and }\qquad G_2(n,t)=-2x(1-x)tn+r_2(t).
$$
Let $r(t)=r_1(t)+r_2(t)\log_23$ and let $T'$ be such that for all $2\leq t'\leq K'$, $|r(t)|\leq T'$ .  Then for all $2\leq t\leq K'$,
$$
G_1(n,t)+G_2(n,t)\log_23\leq tn(1-x^2-2x(1-x)\log_23)+T'.
$$
By (\ref{line5}), $1-x^2-2x(1-x)\log_23<0$, so we can choose $M$ sufficiently large so that if $n>M$, then $n(1-x^2-2x(1-x)\log_23)+T'<0$.  Then for all $2\leq t\leq K'$ and $n\geq M, t$, 
$$
\frac{k(n,t)}{f(n,t)}\leq 2^{nt(1-x^2-2x(1-x)\log_23)+T'}<2^{n(1-x^2-2x(1-x)\log_23)+T'}<1.
$$
Thus $\frac{k(n,t)}{f(n,t)}<1$ for all $n\geq \max\{M, K'\}$ and $2\leq t\leq n$.  
\qed

\bibliography{refs.bib}

\bibliographystyle{amsplain}

\end{document}